\documentclass[10 pt]{amsart}
\usepackage{graphicx}
\usepackage{tikz}
\usepackage{tikz-cd}
\usepackage{float}
\usepackage{tabularx}
\usepackage{mathtools}
\usepackage{ amsbsy} 
\usepackage{amsmath}
\usepackage{amssymb}
\usepackage[utf8]{inputenc}
\usepackage{graphicx}%
\usepackage{multirow}%
\usepackage{amsmath,amssymb,amsfonts}%
\usepackage{amsthm}
\usepackage{mathrsfs}%
\usepackage[title]{appendix}%
\usepackage{xcolor}%
\usepackage{textcomp}%
\usepackage{manyfoot}%
\usepackage{booktabs}%
\usepackage{listings}
\usepackage{amssymb}
\usepackage{amsfonts}
\usepackage{amsmath}
\usepackage{epsfig}
\usepackage{indentfirst}
\usepackage[usenames,dvipsnames]{pstricks}
\usepackage{pst-grad}
\usepackage{pst-plot}
\usepackage[margin=1.0in]{geometry}
\usepackage{graphicx}
\usepackage{tikz}
\usepackage{tikz-cd}
\usepackage{float}
\usepackage{tabularx}
\usepackage{mathtools}
\usepackage{todonotes}

 \newtheorem{theorem}{Theorem}[section]
 \newtheorem{corollary}[theorem]{Corollary}
 \newtheorem{lemma}[theorem]{Lemma}
 \newtheorem{proposition}[theorem]{Proposition}
 \theoremstyle{definition}
 \newtheorem{definition}[theorem]{Definition}
 \theoremstyle{remark}
 \newtheorem{remark}[theorem]{Remark}
 
 \numberwithin{equation}{section}

\DeclareMathOperator {\Hom}{Hom}
\DeclareMathOperator {\End}{End}
\DeclareMathOperator {\Ind}{Ind}
\DeclareMathOperator {\ind}{ind}
\DeclareMathOperator {\Res}{Res}

\DeclareMathOperator {\Ext}{Ext}
\DeclareMathOperator {\Stab}{Stab}
\DeclareMathOperator {\Gl}{GL}
\DeclareMathOperator {\arrow}{arrow}

\DeclareMathOperator{\sgn}{sgn}

\newcommand{\Z}{\mathbb{Z}}

\newcommand{\B}{\mathcal{B}}
\newcommand{\W}{\mathcal{W}}

\newcommand{\Sg}{\mathfrak{S}}
\newcommand{\h}{\mathcal{H}}
\newcommand{\blam}{\pmb{\lambda}}
\newcommand{\bmu}{\pmb{\mu}}
\newcommand{\banu}{\pmb{\nu}}
\newcommand{\bal}{\pmb{\alpha}}

\newcommand{\bosig}{\pmb{\sigma}}
\newcommand{\bolam}{\pmb{\lambda}}
\newcommand{\bomu}{\pmb{\mu}}
\newcommand{\bonu}{\pmb{\nu}}

\newcommand{\Lam}{\pmb{\Lambda}}
\newcommand{\ta}{\mathfrak{t}}
\newcommand{\tab}{\pmb{\mathfrak{t}}}
\newcommand{\F}{\mathcal{F}}

\newcommand{\ch}{\mathrm{char}}

\begin{document}

\renewcommand{\baselinestretch}{1.1}
\title[Construction of Young modules and filtration multiplicities for Brauer algebras of type $C$] {Construction of Young modules and filtration multiplicities for Brauer algebras of type $C$}

\author{Sulakhana Chowdhury}\address{Indian Institute of Science Education and Research Thiruvananthapuram, Thiruvananthapuram, \newline
Email: sulakhana17@iisertvm.ac.in}
\author{Geetha Thangavelu}\address{Indian Institute of Science Education and Research Thiruvananthapuram, Thiruvananthapuram, \newline
Email: tgeetha@iisertvm.ac.in}


\begin{abstract}
 In this paper, we construct the permutation modules and Young modules for Brauer algebras of type $C$ by extending the representation theory of the group algebra of hyperoctahedral groups. Additionally, we develop a stratifying system for Brauer algebras of type $C$, thereby extending the work  of Hemmer-Nakano in \cite{HN} on Hecke algebras. This framework allows us to determine when the multiplicities of cell modules in any filtration are well-defined. As a result, we prove that if the characteristic of the field is neither $2$ nor $3$, then every permutation module of the Brauer algebra of type $C$ decomposes into a direct sum of indecomposable Young modules. We also establish certain cohomological criteria for the group algebra of the hyperoctahedral groups, which are necessary to prove the results for the Brauer algebras of type $C$.
\end{abstract}

\maketitle

\section{Introduction}
 Schur-Weyl duality is a fundamental result in representation theory that links the representations of algebraic groups with those of finite-dimensional algebras. Over an algebraically closed field $K$ of characteristic $0$, Schur \cite{Schur1927} showed that the diagonal action of the general linear group $\Gl_n$ on the tensor space $V^{\otimes r}$, where $\dim V=n$, commutes with the place-permutation action of the symmetric group $\Sg_r$. When this action of $\Gl_n$ is restricted to the orthogonal group, the Brauer algebra $\B_{r}(\delta)$ naturally appears as the corresponding centralizer  algebra on the other side of the duality. In recent years, there has been growing interest in studying various types of Brauer algebras, such as walled Brauer algebras, Brauer algebras of type $C$, cyclotomic Brauer algebras, and $A$-Brauer algebras, among others. 
 
 In this article, we focus on the representation theory of Brauer algebras of type $C$, denoted by $\B(C_r,\delta)$, where $\delta$ is a distinguished element in $K$. It is well known that the Coxeter group of type $C_r$ can be derived from the Coxeter group of type $A_{2r-1}$ as the subgroup consisting of elements fixed by certain Dynkin diagram automorphism.
 In \cite{CLY}, Cohen, Liu, and Yu introduced the Brauer algebra of type $C$ as a subalgebra of $\B_{2r}(\delta)$, generated by Brauer diagrams that are invariant under this automorphism. We refer to these diagrams as \textit{symmetric diagrams}. 
 
 The representation theory of $\B(C_r, \delta)$ is not yet fully understood. In particular, whether $\B(C_r, \delta)$ satisfies the double centralizer property remains an open question. In this article, an analogous question is addressed in the modular setting: when the tensor space is replaced by permutation modules of $\B(C_r, \delta)$, does there exist a family of modules that provides all the indecomposable summands of these permutation modules over a field of characteristic $p$?
 
 The permutation modules and Specht modules of symmetric group are discussed in \cite{Grb} and \cite{JaB}.  In \cite{Ja}, James proved the decomposition of permutation modules of symmetric groups into a direct sum of indecomposable Young modules, while Hartman and Paget carried out a similar analysis for Brauer algebras in \cite{HP}. Later, Paul extended these results to cellularly stratified algebras, whose  input algebras are isomorphic to group algebra of symmetric groups or their Hecke algebras, and obtained an analogous decomposition for partition algebras in \cite{In}. Recently, the authors of this article proved a similar type of decomposition for walled Brauer algebras in \cite{CGwalled}. In this article, our aim is to generalize these results to the case of $\B(C_r, \delta)$. A key difficulty in this setting is that $\B(C_r,\delta)$ contains the group algebra of hyperoctahedral groups, denoted by $\W_r$, as a subalgebra, which has a more complex structure than the group algebra of symmetric groups and their cohomological properties are not yet fully understood.

In \cite{HN}, Hemmer and Nakano made a significant contribution to the theory of Specht modules by proving that if the underlying field is algebraically closed, and its characteristic is neither $2$ nor $3$, then the multiplicities of Specht modules in any filtration remain the same. The cohomological results established in this article provide a new instance of this phenomenon. In particular, we present analogous results for $\W_r$ in Section \ref{Vanishing cohomology on hyperoctahedral group for hyperoctahedral group}, and for $\B(C_r, \delta)$ in Section \ref{Cell filtration multiplicity for type C}. This leads to the structural main result, where we show that the permutation modules of $\B(C_r, \delta)$ decompose into a direct sum of indecomposable Young modules, provided that the characteristic of the field is neither $2$ nor $3$. 
 
The contents of this article are organized as follows. In Section \ref{Representation theory of hyperoctahedral group for type C}, we begin by reviewing the construction of the permutation modules, Specht modules, and Young modules of $\W_r$. In Corollary \ref{dual Specht module is isomorphic to tensor product of two specht in any field for hyperoctahedral group}, we observe that the dual Specht module of $\W_r$ exhibits a property analogous to that found in group algebra of symmetric groups. Furthermore, in Section \ref{Vanishing cohomology on hyperoctahedral group for hyperoctahedral group}, we establish criteria for the non-vanishing cohomology of both Specht, and dual Specht modules for $\W_r$, using the Morita equivalence of $\W_r$ from \cite{DJ92}. Section \ref{Brauer algebras of type C for type C} provides the basic definitions and the cellular structure of $\B(C_r,\delta)$. In Section \ref{Construction of Young module for type C}, we construct the Young modules of $\B(C_r,\delta),$ and show that they possess properties similar to those of Brauer algebras. 

In the Section \ref{Cell filtration multiplicity for type C}, we present the cohomological results of $\B(C_r,\delta)$. In particular, we prove in Lemma \ref{Hom between two cell modules vanishing condition for type C} and Lemma \ref{Ext between two cell modules vanishing condition for type C} that the cohomology of the cell modules for $\B(C_r,\delta)$ is non-vanishing, provided that the characteristic of the field is neither $2$ nor $3$. Consequently, we show that the permutation modules and their direct summands admit a cell filtration (Theorem \ref{cell filtration of permutation module for type C}), and are relatively projective (Corollary \ref{relative projective of the permutation module and its direct summand for type C}) in the category of $\B(C_r,\delta)$-modules that possess a cell filtration. Finally, in Section \ref{Main theorem for type C}, we prove our main result in Theorem \ref{Decomposition of the permutation module for Brauer algebras of type C for type C}, which establishes the decomposition of permutation modules of $\B(C_r,\delta)$ into a direct sum of indecomposable Young modules, provided that the characteristic of the field is neither $2$ nor $3$.


\section{Representation theory of hyperoctahedral group} \label{Representation theory of hyperoctahedral group for type C}

For a natural number $r$, let $W_r$ denote the hyperoctahedral group, also known as the Weyl group of type $B_r$, which has several equivalent descriptions:

\begin{itemize}
    \item[1.]$W_r$ is generated by elements $s_0, s_1, \cdots, s_{r-1}$, subject to the relations: $s_k^2=1$ and $(s_is_j)^{m_{ij}}=1$, where $0 \leq k \leq r-1$, $m_{01}=4$, $m_{ij}=2$ if $|i-j|\geq 2$, and $m_{ij}=3$ if $j=i+1$ for $0<i <r-1$.
    \item[2.] $W_r$ is isomorphic to the wreath product $\Z/2\Z\wr \Sg_r$, where $\Sg_r$ is the symmetric group on $r$ elements. 
    \item[3.] $W_r$ is the group of all orthogonal transformations on $\mathbb{R}^r$ that permute and change signs  of elements in an orthonormal basis.
\end{itemize}
Thus, any element $\bosig \in W_r$ can be expressed as $
   \bigl(\begin{smallmatrix}
    -r & -(r-1) & \cdots & -1 & 1 & \cdots & (r-1) & r \\
   -i_1 & -i_2 & \cdots & -i_r & i_r & \cdots & i_2 & i_1 
  \end{smallmatrix}\bigr).$
For simplicity, we write $$\bosig = (-i_1,\cdots, -i_r,i_r, \cdots, i_1).$$ Alternatively, $\bosig$ can be written as  $\bosig = \Big(\overset{r}{\underset{i=1}{\prod}} \xi_i \Big) \sigma$, where $\sigma \in \Sg_r$ and each $\xi_i$ is defined by
\[
\xi_i(j)=\begin{cases}\pm i &\text{ if }j=\mp i,\\ j &\text{  otherwise.}\end{cases}
\]
 
Let $t_1=s_0$ and for $2\leq i \leq r$, define $t_i=s_{i-1}t_{i-1}s_{i-1}$. It follows that $t_i^2=1$ and $t_it_j=t_jt_i$  if $|i-j|>2$. The subgroup generated by the elements $t_i$, is a normal subgroup isomorphic to $(\Z/2\Z)^r$. Clearly, $W_r/(\Z/2\Z)^r\cong \Sg_r$. By [\cite{GrHyperoctahedral}, Lemma 1.1.1], the group $W_r$ can be embedded into the symmetric group $\Sg_{2r}$. One can visualize elements of $W_r$ as a two-row diagram with $2r$ vertices placed on the top and bottom rows. The vertices are indexed by the set $\{-r, \cdots , -1,1,  \cdots, r\}$ from left to right. Each vertex in the top row is connected to a corresponding vertex in the bottom row by an edge. There is a vertical line between $-1$ and $1$, which we refer to as the \textit{reflecting axis}. Any element in $W_r$ has a mirror image through the reflecting axis. The diagrammatic representation of the generators $s_i$ is shown in Figure \ref{Generators of the hyperoctahedral group}. 

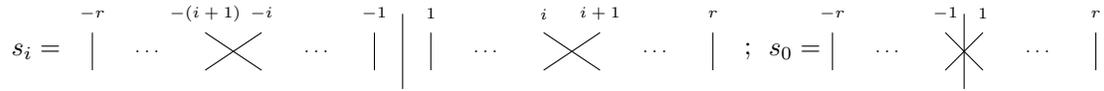
\begin{figure}[H]
\begin{center}
\begin{tikzpicture}[x=0.75cm,y=0.5cm]
\draw[-] (0,0)-- (0,-1);
\draw[-] (2,0)-- (3,-1);
\draw[-] (3,0)-- (2,-1);
\draw[-] (5,0)-- (5,-1);
\draw[-] (5.5,0.5)-- (5.5,-1.5);
\draw[-] (6,0)-- (6,-1);
\draw[-] (8,0)-- (9,-1);
\draw[-] (9,0)-- (8,-1);
\draw[-] (11,0)-- (11,-1);
\node at (-1,-0.5) {$s_{i}=$};
\node at (0,0.5) {\tiny $-r$};
\node at (1,-0.5) {\tiny $\cdots$};
\node at (2,0.5) {\tiny $-(i+1)$};
\node at (3,0.5) {\tiny $-i$};
\node at (4,-0.5) {\tiny $\cdots$};
\node at (5,0.5) {\tiny $-1$};
\node at (6,0.5) {\tiny $1$};
\node at (7,-0.5) {\tiny $\cdots$};
\node at (8,0.5) {\tiny $i$};
\node at (9,0.5) {\tiny $i+1$};
\node at (10,-0.5) {\tiny $\cdots$};
\node at (11,0.5) {\tiny $r$};
\end{tikzpicture}
\begin{tikzpicture}[x=0.5cm,y=0.5cm]
\draw[-] (0,0)-- (0,-1);
\draw[-] (3,0)-- (4,-1);
\draw[-] (3.5,0.5)-- (3.5,-1.5);
\draw[-] (4,0)-- (3,-1);
\draw[-] (7,0)-- (7,-1);
\node at (-2.25,-0.5) {$;$};
\node at (-1,-0.5) {$s_{0}=$};
\node at (0,0.5) {\tiny $-r$};
\node at (1.5,-0.5) {\tiny $\cdots$};
\node at (3,0.5) {\tiny $-1$};
\node at (4,0.5) {\tiny $1$};
\node at (5.5,-0.5) {\tiny $\cdots$};
\node at (7,0.5) {\tiny $r$};
\end{tikzpicture}
\caption{Generators of hyperoctahedral group.}\label{Generators of the hyperoctahedral group}
\end{center}
\end{figure}

\subsection{Combinatorics related to $W_r$}

First recall some standard definitions and notations from \cite{Aa77}, \cite{AMP} and \cite{Mo81}. Let $K$ be an algebraically closed field, and the group algebra of $W_r$ be denoted by $\W_r$. Let $a$ be an integer such that $0 \leq a \leq r$. A \textit{composition} $\lambda$ of $r$ is a finite sequence of non-negative integers $(\lambda_1, \cdots, \lambda_k)$ such that $\sum\limits_{i=1}^k \lambda_i=r$. If the sequence is non-increasing, i.e., $\lambda_1 \geq \cdots \geq \lambda_k$, then $\lambda$ is called a \textit{partition} of $r$, denoted by $\lambda \vdash r$. Let $\Lambda_r$ be the collection of all partitions of $r$.  A \textit{bi-composition} $\blam=(\lambda^{1},\lambda^{2})$ of $r$ is a pair, where $\lambda^{1}=( \lambda_1^{1}, \cdots, \lambda_{k_1}^{1})$ is a composition of $a$ and $\lambda^{2}=(\lambda_1^{2}, \cdots, \lambda_{k_2}^{2})$ is composition of $r-a$ such that $\sum\limits_{i=1}^2\sum\limits_{j=1}^{k_i} \lambda_{j}^{i}=r$. We call $\bolam$ a \textit{bi-partition} of $r$ if both $\lambda^{1}$ and $\lambda^{2}$ are partitions. We also allow either $\lambda^{1}$ or $\lambda^{2}$ to be the empty partition (i.e. having zero parts). The \textit{conjugate} $\blam'$ of a bi-partition $\blam$ is obtained by swapping the partitions $\lambda^{1}$, $\lambda^{2}$ and then taking their conjugates separately, that is, $\blam'=(\lambda^{2'},\lambda^{1'})$. Let $\Lam_r$ denote the set of all bi-partitions of $r$. Given two bi-partitions $\blam=\big(\lambda^{1}, \lambda^{2}\big)$ and $ \bmu=\big(\mu^{1},\mu^{2}\big)$ in $\Lam_r$, we define the dominance order ``$\unrhd$" on $\Lam_r$ as follows: we write $\blam \trianglerighteq \bmu $ if 
\begin{align*}
\sum_{i=1}^{j} \lambda_{i}^{1} &\geq \sum_{i=1}^{j} \mu_{i}^{1} \text{ for all }j,
\end{align*}
and
\begin{align*}
|\lambda^{1}| +\sum_{i=1}^{j'} \lambda_i^{2} &\geq |\mu^{1}| + \sum_{i=1}^{j'} \mu_i^{2} \text{ for all } j'.
\end{align*}
Thus, we have $\blam \trianglerighteq \bmu $ if and only if $\bmu' \trianglerighteq \blam'$. If $\blam \trianglerighteq \bmu$ and $\blam \neq \bmu$, then we write $\blam \triangleright \bmu$. 

Given $\lambda \vdash r$, the associated \textit{Young diagram} consisting of $r$ boxes arranged in $k$ rows, where the $i$-th row contains $\lambda_i$ left-justified boxes. A \textit{Young tableau} $\ta$ of shape $\lambda$ is  a Young diagram of shape $\lambda$, where each box is filled with distinct positive integer $1, 2, \cdots,r$. A Young \textit{$\blam$-bi-tableau} $\tab$ associated with $\blam$ is a pair of tableaux $(\ta^{1}, \ta^{2})$, where each $\ta^{i}$ is a $\lambda^{i}$-tableau, for $i=1,2$. Each box of the Young bi-tableau $\tab$ can be filled by the elements from the set $\{-r,\cdots,-1,1, \cdots, r\}$, where $i$ and $-i$ never appear simultaneously. The hyperperoctahedral group $W_r$ acts on $\tab$ by permuting the entries of the $\blam$-bi-tableau. A \textit{row-permutation} of $\tab$ is an element of $W_r$, which permutes the entries within each row of $\tab$, and may change the sign of the entries in $\ta^{2}$. Let $R_{\tab}$ denote the group of all such row permutations of $\tab$. Then we have the following isomorphism:
\begin{equation*}
    R_{\tab} \cong \Sg_{\lambda_1^{1}} \times \cdots \times \Sg_{\lambda_{k_1}^{1}} \times W_{\lambda_1^{2}} \times \cdots \times W_{\lambda_{k_2}^{2}}.
\end{equation*}
A \textit{column permutation} of $\tab$ permutes the elements of each column of $\tab$ and may change the sign of any entry in $\ta^{1}$. Let $C_{\tab}$ denote the group of column permutations of $\tab$. Note that $C_{\tab}= R_{\tab'}$, where $\tab'$ is a Young bi-tableau of shape $\bolam'$. An \textit{equivalent relation} $\sim$ is defined on the set of $\blam$-tableaux by $\tab_1\sim \tab_2$ if there exists $\bosig\in R_{\tab_1}$ such that $\tab_2=  \tab_1 \bosig$. The equivalence class containing $\tab$ is called $\blam$-\textit{bi-tabloid} and is denoted by $\{\ta\}$. 

Consider $M(\blam)$ to be the K-vector space spanned by all $\blam$-bi-tabloids. The algebra $\W_r$ acts on $\blam$-bi-tabloid by $\{\tab\}\cdot \bosig = \{\tab ~\bosig\}$, for $\bosig \in W_r$. By extending this action linearly, $M(\blam)$ becomes a $\W_r$-module. The $\W_r$-module $M(\blam)$ is called the \textit{permutation module} of $\W_r$. The dimension of $M(\blam)$ is given by 
\[
\dim_K M(\blam)= 2^{r-|\lambda^{2}|}\frac{r!}{\lambda_1^{1}! \cdots \lambda_{k_1}^{1}!  \lambda_1^{2}! \cdots \lambda_{k_2}^{2}!}.
\]
Let $W_{\blam}$ denote the subgroup $\Sg_{\lambda_1^{1}} \times \cdots \times \Sg_{\lambda_{k_1}^{1}} \times W_{\lambda_1^{2}} \times \cdots \times W_{\lambda_{k_2}^{2}}$ of $W_r$, and its corresponding group algebra by $\W_{\blam}$. Since $\W_r$ acts transitively on the set of $\blam$-tabloids and the stabilizer of a fixed $\blam$-bi-tabloid $\{\tab\}$ in $\W_r$ is conjugate to the subgroup $W_{\blam}$, we obtain the following isomorphism
of $\W_r$-modules:
\begin{equation*}
    M(\blam) \cong \Ind_{\W_{\blam}}^{\W_r} 1_{\W_{\blam}},
\end{equation*}
where $1_{\W_{\blam}}$ is the trivial $\W_{\blam}$-module. 

The \textit{alternating column sum} $k_{\tab}$ is an element of $W_r$ defined by $k_{\tab}= \sum\limits_{\bosig \in C_{\tab}} (\sgn \bosig) \bosig.$ If $\tab$ is a bi-tableau, then the corresponding $\blam$-\textit{polytabloid} $e_{\tab}$ is defined by $e_{\tab}:= \{\tab\}\cdot k_{\tab}$. For any bi-partition $\blam$, the\textit{ Specht module} $S(\blam)$ is the submodule of $M(\blam)$ spanned by all $\blam$-polytabloids. It is also a cyclic $\W_r$-module generated by a single $\blam$-polytabloid as shown in [\cite{Mo81}, Theorem 2.5]. Define a bilinear form $\langle-, -\rangle$ on the permutation module $M(\blam)$ such that it is orthonormal on the $\blam$-bi-tabloids. That is, if $\{\tab_1\}, \{\tab_2\} \in M(\blam)$, then
\begin{equation*}
    \langle \{\tab_1 \}, \{\tab_2\}\rangle= \delta_{\{\tab_1\} ,\{\tab_2\}}=
    \begin{cases}
        1 & \text{ if } \{\tab_1 \}= \{\tab_2 \},\\
        0 &\text{ otherwise.}
    \end{cases}
\end{equation*}
For any bi-partition $\blam$, the quotient module $\frac{S(\blam)}{S(\blam)\cap (S(\blam))^{\perp}}$ is either zero or absolutely irreducible by [\cite{Mo81}, Theorem 2.12]. We denote this quotient module by $D(\blam)$. When $\mathrm{char}~K =0$, it follows from [\cite{Mo81}, Corollary 2.13] that the set $\{S(\blam): \blam \in \Lam_r\}$ forms a complete set of simple $\W_r$-modules. If $\mathrm{char}~K =p~( \neq 2)$, then a bi-partition $\blam$ is called $p$-\textit{regular} if both of the components $\lambda^{1}$ and $\lambda^{2}$ are $p$-regular. A bi-partition $\blam$ is said to be $2$-\textit{regular} if $|\lambda^{1}|=0$ and $\lambda^{2}$ is $2$-regular. When $\mathrm{char}~K =p$, all simple modules are indexed by the set of $p$-regular bi-partitions of $r$ as shown in [\cite{Mo81}, Theorem 2.17].

If the characteristic of the field is $0$, then $M(\blam)$ can be decomposed into a direct sum of Specht modules. However, in the modular case, the Specht modules are no longer simple. Therefore, our interest lies in classifying the indecomposables in the decomposition of $M(\blam)$ over a field of positive characteristic. By [\cite{CR90}, Theorem 6.12], $M(\blam)$ can be decomposed into a direct sum of indecomposable modules, say $M(\blam)= \bigoplus_i Y_i$.

Let $\tab$ be a $\blam$-tableau with $k_{\tab}$ be its corresponding alternating column sum. Since $M(\blam)k_{\tab}= Ke_{\tab}$ by [\cite{Mo81}, Lemma 2.9], there exists a unique summand $Y_j$ of $M(\bolam)$ such that $S(\blam) \cap Y_j \neq 0$. By [\cite{Mo81}, Theorem 2.11], $Y_j$ is the unique direct summand of $M(\blam)$ containing $S(\blam)$ as a submodule. This module is referred to as the \textit{Young module} of $\W_r$ associated with $\blam$, and is denoted by $Y(\blam)$. The decomposition of permutation module of the Ariki-Koike algebra (which is the quantum analogue of the group algebra of $\Z_m \wr \Sg_r$) is proved in [\cite{Ma}, (3.5)] using the concept of cyclotomic $q$-Schur algebra. However, in our context, we present the following decomposition under the assumption $q=1=Q_1=\cdots=Q_m$ and $m=2$.
\begin{theorem} \label{Decomposition of the permutation module of hyperoctahedral group}
Let $ \blam \in \Lam_r.$ The permutation module $M(\blam)$ of $\W_r$ can be decomposed as a direct sum of indecomposable Young modules 
\begin{equation*}
   M(\blam)= Y(\blam) \oplus  \bigoplus_{ \substack{\bmu \unrhd \blam\\ \bmu \in \Lam_r}}   Y(\bmu)^{ c_{\bmu}},
\end{equation*}
where $c_{\bmu}$ is the multiplicity with which $Y(\bmu)$ appears in the decomposition of $M(\blam)$.
\end{theorem}

Let $M$ be a $\W_r$-module. The dual module $M^*$ is the $K$-vector space spanned by all the linear maps from $M$ to $K$. The action of $\W_r$ on $M^*$ is given by $(m)(\phi \cdot \bosig):= (m \cdot \bosig^{-1}) \phi$, for $m \in M,\bosig \in \W_r \text{ and } \phi \in M^*$. Define the sign function $\sgn: \W_r \longrightarrow \{-1,1\}$ as follows: $\sgn~ \bosig=(-1)^{\sum_i x_i} \sgn~ \sigma$, where $x_i$ is the number of $\xi_i$ appears in $\bosig$. When $\blam=\big((1^r), (0)\big)$, we call $S(\bolam)$ the \textit{sign module}, and is denoted it by $\sgn$. This is an one-dimensional $\W_r$-module, where each element $\bosig \in \W_r$ acts by $\sgn~\bosig$. Let $\blam'$ be the conjugate of $\blam$. The following proposition connects the dual Specht modules of $\W_r$ associated to $\bolam'$, leads to a result similar to that for symmetric groups, as stated in [\cite{JaB}, Theorem 6.7].

\begin{proposition} \label{dual Specht module is isomorphic to Specht tensor sgn over Q for type C}
    Over the field $\mathbb{Q}$, the dual of $S(\blam')$ is isomorphic to $ S(\blam) \otimes_{\mathbb{Q}} \sgn$.
\end{proposition}

\begin{proof}
It is straightforward to observe that $e_{\tab}$ is skew-invariant under the action of $C_{\tab}$. Consequently, $e_{\tab}\otimes 1$ remains invariant under the group $R_{\tab'}=C_{\tab}$, where $\tab'$ is a $\bolam'$-bi-tableau. Hence, there exists a unique $\W_r$-equivariant map 
    \begin{align*}
        f:M(\blam') &\longrightarrow S(\blam) \otimes \sgn\\
        \{\tab'\} &\longmapsto e_{\tab} \otimes 1.
    \end{align*}
 And, 
    \[
    f(e_{\tab'})=f(\{\tab'\} \cdot k_{\tab'})= (e_{\tab} \rho_{\tab})\otimes 1, 
    \]
    where $\rho_{\tab}$ is the row sum of the $\bolam$-bi-tableau $\tab$. Now, we need to show that $e_{\tab} \rho_{\tab}\neq 0$:
    \[
    \langle e_{\tab} \rho_{\tab},\tab \rangle= \langle e_{\tab}, \tab \cdot \rho_{\tab} \rangle= |R_{\tab}|\langle e_{\tab}, \tab \rangle =|R_{\tab}|\neq 0.
    \]
Therefore, we conclude that the restriction $f|_{S(\blam')}: S(\blam') \longrightarrow S(\blam) \otimes \sgn$ defines a non-zero $\W_r$-module homomorphism. Given that both modules are irreducible and non-zero,  Schur's lemma implies that $f|_{S(\blam')}$ is infact an isomorphism.

\end{proof}

We denote the dual of the Specht module $S(\blam)$ by $S'(\blam)$, and refer to it as the \textit{dual Specht module}. 
\begin{corollary} \label{dual Specht module is isomorphic to tensor product of two specht in any field for hyperoctahedral group}
    Over any field of characteristic $p$, the dual Specht module $S'(\blam')$ is isomorphic to $S(\blam) \otimes S\big((1^r),(0)\big)$. 
\end{corollary}
\begin{proof}
    The proof follows immediately from the fact that $S'(\blam')$ over $K$ is isomorphic to $K\otimes_{\mathbb{Q}} S'(\blam')$. 
\end{proof}

\subsection{Morita equivalence for hyperoctahedral groups}\label{Morita equivalence for Hecke algebras for type C}

The previous section examined the Specht module and its dual via tableau combinatorics. This subsection establishes their connection to the representation theory of symmetric groups.

Let $0 \leq b \leq r$. Recall from [\cite{DJ92}, Definition 3.27] that $e_{b,r-b}:=\tilde{z}_{r-b,b}^{-1}v_{b,r-b} h_{r-b,b}$ is the orthogonal idempotent of $\W_r$, where $\tilde{z}_{r-b,b}$ is an invertible element of $K\Sg_{b,r-b}$ provided that $\ch~K \neq 2$, and $v_{b,r-b}$, $h_{b,r-b}$ are elements in $\W_r$ (see [\cite{DJ92}, Definition 3.8, Definition 2.3]). Here $K\Sg_{b,r-b}$ denotes the group algebra of direct product of two symmetric groups. If $\ch~K \neq 2$, then an idempotent $e_{b,r-b}$ satisfies 
 \begin{equation}\label{Morita equivalent idempotent satisfy the vi condition for type C}
    v_{b,r-b} \W_r = e_{b,r-b}\W_r,  \text{ for }0\leq b \leq r.
 \end{equation}
 It follows from [\cite{DJ92}, Lemma 4.16] that $e_{b,r-b}\W_r e_{b,r-b}$ is canonically isomorphic to $K\Sg_{b,r-b}$ for each $b$. Let $\epsilon= \sum_{b=0}^{r} e_{b,r-b}$, then $\epsilon \W_r$ is a projective generator for the category of $\W_r$-modules. Therefore, the category of $\W_r$-modules is Morita equivalent to the category of right $\epsilon\W_r \epsilon$-modules. This gives rise to the following Morita equivalence:
\begin{equation}\label{Morita equivalent of Wn for type C}    
F:\text{\textbf{mod}-}\W_r \underset{\text{Morita}}{\simeq} \text{\textbf{mod}-}\bigoplus_{b=0}^r K\Sg_{b,r-b}:G,
\end{equation}
where $\text{\textbf{mod}-}\W_r$ (resp. $\text{\textbf{mod}-}\bigoplus_{b=0}^r K\Sg_{b,r-b}$) denotes the category of (finite dimensional) right $\W_r$-modules (resp. $\bigoplus_{b=0}^r K\Sg_{b,r-b}$-modules). The equivalence is realized by the functors
\[
F(N)=N \otimes_{\W_{r}} \W_r\epsilon \quad \text{ and } \quad G(M)=\Hom_{\W_r}( \epsilon\W_r, M),
\]
for $N \in\text{\textbf{mod}-}\W_r$ and $M\in \text{\textbf{mod}-}\bigoplus_{b=0}^r K\Sg_{b,r-b}$. These functors satisfy the identities $$F \circ G=1_{\text{\textbf{mod}-}\bigoplus_{b=0}^r K\Sg_{b,r-b}} \text{ and }G \circ F= 1_{\text{\textbf{mod}-}\W_r}.$$

\subsection{Vanishing cohomology on hyperoctahedral group} \label{Vanishing cohomology on hyperoctahedral group for hyperoctahedral group}

We generalize the results of Hemmer and Nakano \cite{HN} to $\W_r$ establishing vanishing results for the cohomology of Specht and dual Specht modules of $\W_r$. Our approach relies on the Morita equivalence of Dipper and James \cite{DJ92}, which relates $\W_r$-modules to $\bigoplus\limits_{b=0}^r K\Sg_{b,r-b}$-modules (see Section \ref{Morita equivalence for Hecke algebras for type C}).

For $0\leq b \leq r$, define the subset $\Lam_b^{+}=\{\bolam=(\lambda^1, \lambda^2)\in \Lam_r: |\lambda^1|=b\}$. There is a natural bijection 
\[
\Lambda_b \times \Lambda_{r-b} \longleftrightarrow  \Lam_b^{+},
\]
sending the pair $\lambda^1 \in \Lambda_b$ and $\lambda^2 \in \Lambda_{r-b}$ to the bi-partition $\bolam \in \Lam_b^{+}$. The following result due to Dipper and Mathas [\cite{DM02}, Proposition 4.11] describes the image of Specht modules under the Morita equivalence.

\begin{proposition}\label{Image of Specht under Morita is the tensor product of Specht}
    Let $\ch ~K \neq 2$. If $\blam =(\alpha, \beta)\in \Lam_r$, where $\alpha \in \Lambda_b$ and $\beta \in \Lambda_{r-b}$, then we have $S(\blam)= G(S^{\alpha} \otimes S^{\beta})$. 
\end{proposition}

The following corollary is an immediate consequence of Proposition \ref{Image of Specht under Morita is the tensor product of Specht}. 


\begin{corollary}\label{Image of Specht with (n,0) under Morita is the tensor product of Specht}
Let $\ch~K \neq 2$ and $0 \leq b\leq r$. Then we have 
\begin{itemize}
\item[(i)] $S\big((r),(0)\big)= G(S^{(b)} \otimes S^{(r-b)})$.
\item[(ii)]$S\big((1^r),(0)\big)= G(S^{(1^b)} \otimes S^{(1^{r-b})})$.
\end{itemize}
\end{corollary}

\begin{corollary}\label{Image of dual Specht under Morita is the tensor product of dual Specht}
    Let $\ch~K \neq 2$. If $\blam=(\alpha, \beta) \in \Lam_r$, where $\alpha \in \Lambda_b$ and $\beta \in \Lambda_{r-b}$, then we have $S'(\blam)= G(S_{\alpha} \otimes S_{\beta})$. 
\end{corollary}
\begin{proof}
By Corollary \ref{dual Specht module is isomorphic to tensor product of two specht in any field for hyperoctahedral group}, we have an isomorphism $S'(\blam)\cong S(\blam) \otimes_K \sgn$. Applying Proposition \ref{Image of Specht under Morita is the tensor product of Specht} together with Corollary \ref{Image of Specht with (n,0) under Morita is the tensor product of Specht} (ii), we obtain
 \[
 S(\blam) \cong G\big(S^{\alpha}\otimes S^{\beta}) \text{ and } S\big((1^r),(0)\big) \cong G(S^{(1^b)}\otimes S^{(1^{r-b})}).
 \]
Since the functor $G$ preserve the tensor products over $K$, it follows that
\[
S'(\blam)= G\big((S^{\alpha} \otimes S^{(1^b)})\otimes ( S^{\beta}\otimes S^{1^{(r-b)}} )\big).
\]
The result then follows from [\cite{JaB}, Theorem 8.15].
\end{proof}
  
The following theorem and propositions extend the results of \cite{HN}.
\begin{theorem}\label{Ext between trivial and dual Specht module is zero for hyperoctahedral group}
      If $\mathrm{char}~K \neq 2,3$, then $\Ext_{\W_r}^1 (k, S'(\blam))=0$, where $\blam$ is a bi-partition of $r$. 
\end{theorem}
\begin{proof}
Since the trivial module $k$ is isomorphic to $S\big((r),(0)\big)$, it follows from (\ref{Morita equivalent of Wn for type C}) that
\begin{equation*}\label{Ext ist line using Morita for type C}
\Ext_{\W_r}^1\Big(k,~S'(\blam)\Big) \cong \Ext_{\overset{r}{\underset{b=0}{\oplus}}K\Sg_{b} \times K\Sg_{r-b}}^1\Big(F(S\big( (r), (0)\big),~F\big(S'(\blam)\big)\Big).
\end{equation*}
Using Corollary \ref{Image of Specht with (n,0) under Morita is the tensor product of Specht} (i) and Corollary \ref{Image of dual Specht under Morita is the tensor product of dual Specht}, together with the natural isomorphism between the functors $F$ and $G$, we obtain
\begin{align}\label{Ext Kunneth formula for type C}
\Ext_{\overset{r}{\underset{b=0}{\oplus}}K\Sg_{b} \times K\Sg_{r-b}}^1\Big(F(S\big( (r), (0)\big),~F\big(S'(\blam)\big)\Big)&\cong \bigoplus_{b=0}^r \Ext_{K\Sg_{b} \times K\Sg_{r-b}}^1 \Big(S^{(b)} \otimes S^{(r-b)},~S_{\alpha} \otimes S_{\beta} \Big)\nonumber \noindent\\ 
&\cong\bigoplus_{b=0}^r \big[ \big( \Ext_{K\Sg_{b}}^1 (S^{(b)}, S_{\alpha}) \otimes \Hom_{K \Sg_{r-b}} (S^{(r-b)}, S_{\beta}) \big)\nonumber \noindent\\
        & \hspace{1 cm}  \oplus \big(\Hom_{K\Sg_{b}} (S^{(b)}, S_{\alpha}) \otimes \Ext_{K \Sg_{r-b}}^1 (S^{(r-b)}, S_{\beta})\big)\big], 
\end{align}
where the last isomorphism follows from [\cite{CE}, Chapter XI, Theorem 3.1]. By [\cite{HN}, Theorem 3.3.4], when $\ch~K \neq2,3$, we have  $$\Ext_{K\Sg_b}^1(S^{(b)}, S_{\alpha})=0 \text{ and } \Ext_{K\Sg_{r-b}}^1(S^{(r-b)}, S_{\beta})=0. $$ Thus, each summand in (\ref{Ext Kunneth formula for type C}) vanishes, completing the proof.
\end{proof}

The following two propositions were proved for the case $m=2$ in [\cite{RuThesis}, Theorem 10.1.1 and 10.1.3]. While the above proof relied on cellular techniques, our approach is based on the Morita equivalence of $\W_r$ (see Subsection \ref{Morita equivalence for Hecke algebras for type C}). These propositions are instrumental in Section \ref{Cell filtration multiplicity for type C}, where we establish a criteria for the vanishing of the cohomology of cell modules of Brauer algebras of type $C$.

\begin{proposition}\label{Hom between Specht module of hyperoctahedral group for type C}
    Let $\ch~K \neq 2$, and $\blam,\bmu \in \Lam_r$. If $\blam \ntrianglerighteq \bmu$, then $\Hom_{\W_r} \big( S(\blam), S(\bmu)\big)=0$.
\end{proposition}

\begin{proof}
If $\ch~K \neq 2$ and $\blam , \bmu \in \Lam_r$, then by Proposition \ref{Image of Specht under Morita is the tensor product of Specht}, we have 
\[
S(\blam)= G(S^{\alpha} \otimes S^{\beta}) \text{ and } S(\bmu)= G(S^{\gamma}\otimes S^{\nu}),
\] 
where $ \alpha , \gamma \in \Lambda_b$ and $\beta , \nu \in \Lambda_{r-b}$. Applying the Morita equivalence given in (\ref{Morita equivalent of Wn for type C}), we obtain the following isomorphism 
\begin{align*}
 \Hom_{\W_r} \big( S(\blam), S(\bmu)\big)&\cong\bigoplus_{b=0}^r \Hom_{K\Sg_b \times K\Sg_{r-b}} \big( S^{\alpha} \otimes S^{\beta}, S^{\gamma} \otimes S^{\nu}\big)\\
        &\cong \bigoplus_{b=0}^r \Big(\Hom_{K\Sg_b} ( S^{\alpha}, S^{\gamma})\otimes  \Hom_{ K\Sg_{r-b}} (S^{\beta},  S^{\nu} ) \Big),
\end{align*}
where the last isomorphism follows from [\cite{Xi00}, Lemma 3.2]. Since $\blam \ntrianglerighteq \bmu$ implies that at least $\alpha \ntrianglerighteq \gamma$ or $\beta \ntrianglerighteq \nu$, we conclude from [\cite{HN}, Lemma 4.1.1 (i)] (under the assumption $\ch~K \neq 2$) that
\[
\Hom_{K\Sg_b} ( S^{\alpha}, S^{\gamma})=0 \text{ and }\Hom_{ K\Sg_{r-b}} (S^{\beta},  S^{\nu} )=0.
\]
This completes the proof.
\end{proof}
\begin{proposition}\label{Ext between Specht module of hyperoctahedral group for type C}
    Let $\ch~K \neq 2,3$ and $\blam,\bmu \in \Lam_r$. If $\blam \ntriangleright \bmu$, then $\Ext_{\W_r}^1\big( S(\blam), S(\bmu)\big)=0$.
\end{proposition}

\begin{proof}
Since $\ch~K \neq 2,3$ and $\blam,\bmu \in \Lam_r$, we apply the Morita equivalence from (\ref{Morita equivalent of Wn for type C}) to obtain 
\[
 \Ext_{\W_r}^1 \big( S(\blam), S(\bmu)\big)\cong \bigoplus_{b=0}^r \Ext_{K\Sg_b \times K\Sg_{r-b}}^1 \big( S^{\alpha} \otimes S^{\beta}, S^{\gamma} \otimes S^{\nu}\big).
\]
Here $S(\blam) \cong G(S^{\alpha} \otimes S^{\beta})$ and $S(\bmu) \cong G(S^{\gamma} \otimes S^{\nu})$ for some $ \alpha, \gamma \in \Lambda_b$ and $\beta , \nu \in \Lambda_{r-b}$, as given by Proposition \ref{Image of Specht under Morita is the tensor product of Specht}. Using [\cite{CE}, Chapter XI, Theorem 3.1], we have  
\begin{align*}
\Ext_{K\Sg_b \times K\Sg_{r-b}}^1 ( S^{\alpha} \otimes S^{\beta}, S^{\gamma} \otimes S^{\nu}) \cong &\bigoplus_{b=0}^r \big[\big(\Ext_{K\Sg_b}^1 \big( S^{\alpha}, S^{\gamma}) \otimes  \Hom_{ K\Sg_{r-b}} (S^{\beta},  S^{\nu} ) \big) \\
        & \hspace{1cm} \oplus \big(\Hom_{K\Sg_b} \big( S^{\alpha}, S^{\gamma}) \otimes \Ext_{ K\Sg_{r-b}}^1 (S^{\beta},  S^{\nu} ) \big)\big].
\end{align*}
Since $\blam \ntriangleright \bmu$, it follows that at least $\alpha \ntriangleright \gamma$ or $\beta \ntriangleright \nu$. By [\cite{HN}, Lemma 4.2.1], we conclude that 
\[
\Ext_{K\Sg_b}^1 \big( S^{\alpha}, S^{\gamma})=0 \text{ and } \Ext_{ K\Sg_{r-b}}^1 (S^{\beta},  S^{\nu} )=0.
\]
Thus, the proof is complete. 
\end{proof}

\section{Brauer algebras of type $C$} \label{Brauer algebras of type C for type C}

The Brauer algebra of type $C$ was first introduced by Cohen, Liu, and Yu in \cite{CLY}. It arises as a subalgebra of Brauer algebra of type $A_{2r-1}$, consisting specifically of those elements that remain invariant under a diagram automorphism of the Brauer algebra of type $A_{2r-1}$ (see \cite{Bo}).

Throughout this article, we assume that $K$ to be an algebraically closed field of characteristic $p$, which contains a distinguished element $\delta$. For a given $r \in \mathbb{N}$, a \textit{Brauer diagram of type $C$} is a two-row diagram consisting of $2r$ vertices, arranged along the top and bottom rows. The vertices are indexed from left to right as $-r, -(r-1), \cdots, -1, 1, \cdots, r-1, r$. Each pair of vertices is connected by either a vertical edge or a horizontal edge. These diagrams exhibit a mirror symmetry with respect to the vertical axis between the vertices labeled by $-1$ and $1$. In other words, every horizontal and vertical edge in the diagram has a corresponding reflected edge across the axis. One such diagram is shown in Figure \ref{Brauer diagram of type C}. 
\begin{center}
\begin{figure}[H]
\begin{tikzpicture}[x=0.65cm,y=0.65cm]
\draw[-] (0,0)-- (3,-1);
\draw[-] (1,0)-- (0,-1);
\draw[-] (2,0).. controls(3.5,-0.5)..(5,0);
\draw[-] (3,0).. controls(4.5,-0.5)..(6,0);
\draw[-] (4,0.5)-- (4,-1.5);
\draw[-] (7,0)-- (8,-1);
\draw[-] (8,0)-- (5,-1);
\draw[-] (1,-1).. controls(4,-0.5)..(7,-1);
\draw[-] (2,-1).. controls(4,-0.75)..(6,-1);
\node at (0,0.5) {\tiny $-4$};
\node at (1,0.5) {\tiny $-3$};
\node at (2,0.5) {\tiny $-2$};
\node at (3,0.5) {\tiny $-1$};
\node at (5,0.5) {\tiny $1$};
\node at (6,0.5) {\tiny $2$};
\node at (7,0.5) {\tiny $3$};
\node at (8,0.5) {\tiny $4$};
\end{tikzpicture}
\caption{Brauer diagram of type $C$ in $\B(C_4, \delta)$.} \label{Brauer diagram of type C}
\end{figure}
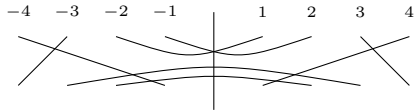
\end{center}
The Brauer algebra of type $C$, denoted by $\B(C_r,\delta)$, 
is an associative unital $K$-algebra with a basis consisting of all Brauer diagrams of type $C$. This algebra is generated by $\{s_i: i=0, \cdots, r-1\}$ and $\{ e_i: i=0,\cdots, r-1\}$, subject to the relations defined in [\cite{CLY}, Definition 2.1]. The diagrammatic description of the generators of $\B(C_r, \delta)$ is given in Figure \ref{Generators of Brauer diagram of type C}. Note that the subalgebra generated by $e_i$ for $i=0,\cdots,r-1$, is isomorphic to the Temperley–Lieb algebra of type $B_r$, as defined by Dieck in \cite{Dieck}. Similarly, the subalgebra generated by $\{s_i: i=0, 1, \cdots, r-1\}$ is isomorphic to the group algebra of the hyperoctahedral group $\W_r$. The multiplication of any two Brauer diagrams of type $C$ follows the same rules as the multiplication of two Brauer diagrams.
\begin{center}
\begin{figure}[H]
\begin{tikzpicture}[x=0.65cm,y=0.65cm]
\draw[-] (0,0)-- (0,-1);
\draw[-] (2,0)-- (3,-1);
\draw[-] (3,0)-- (2,-1);
\draw[-] (5,0)-- (5,-1);
\draw[-] (6,0.5)-- (6,-1.5);
\draw[-] (7,0)-- (7,-1);
\draw[-] (9,0)-- (10,-1);
\draw[-] (10,0)-- (9,-1);
\draw[-] (12,0)-- (12,-1);
\node at (-0.75,-0.5) {$s_i=$};
\node at (0,0.5) {\tiny $-r$};
\node at (1,-0.5) {\tiny $\cdots$};
\node at (2,0.5) {\tiny $-(i+1)$};
\node at (3,0.5) {\tiny $-i$};
\node at (4,-0.5) {\tiny $\cdots$};
\node at (5,0.5) {\tiny $-1$};
\node at (7,0.5) {\tiny $1$};
\node at (8,-0.5) {\tiny $\cdots$};
\node at (9,0.5) {\tiny $i$};
\node at (10,0.5) {\tiny $i+1$};
\node at (11,-0.5) {\tiny $\cdots$};
\node at (12,0.5) {\tiny $r$};
\node at (12.5,-0.5) {;};
\end{tikzpicture}
\begin{tikzpicture}[x=0.65cm,y=0.65cm]
\draw[-] (0,0)-- (0,-1);
\draw[-] (2,0)-- (2,-1);
\draw[-] (3,0)-- (5,-1);
\draw[-] (4,0.5)-- (4,-1.5);
\draw[-] (5,0)-- (3,-1);
\draw[-] (6,0)-- (6,-1);
\draw[-] (8,0)-- (8,-1);
\node at (-0.75,-0.5) {$s_0=$};
\node at (0,0.5) {\tiny $-r$};
\node at (1,-0.5) {\tiny $\cdots$};
\node at (2,0.5) {\tiny $-2$};
\node at (3,0.5) {\tiny $-1$};
\node at (5,0.5) {\tiny $1$};
\node at (6,0.5) {\tiny $2$};
\node at (7,-0.5) {\tiny $\cdots$};
\node at (8,0.5) {\tiny $r$};
\end{tikzpicture}
\begin{tikzpicture}[x=0.65cm,y=0.65cm]
\draw[-] (0,0)-- (0,-1);
\draw[-] (2,0).. controls(2.5,-0.5).. (3,0);
\draw[-] (2,-1).. controls(2.5,-0.5).. (3,-1);
\draw[-] (5,0)-- (5,-1);
\draw[-] (6,0.5)-- (6,-1.5);
\draw[-] (7,0)-- (7,-1);
\draw[-] (9,0).. controls(9.5,-0.5).. (10,0);
\draw[-] (9,-1).. controls(9.5,-0.5).. (10,-1);
\draw[-] (12,0)-- (12,-1);
\node at (-0.75,-0.5) {$e_i=$};
\node at (0,0.5) {\tiny $-r$};
\node at (1,-0.5) {\tiny $\cdots$};
\node at (2,0.5) {\tiny $-(i+1)$};
\node at (3,0.5) {\tiny $-i$};
\node at (4,-0.5) {\tiny $\cdots$};
\node at (5,0.5) {\tiny $-1$};
\node at (7,0.5) {\tiny $1$};
\node at (8,-0.5) {\tiny $\cdots$};
\node at (9,0.5) {\tiny $i$};
\node at (10,0.5) {\tiny $i+1$};
\node at (11,-0.5) {\tiny $\cdots$};
\node at (12,0.5) {\tiny $r$};
\node at (12.5,-0.5) {;};
\end{tikzpicture}
\begin{tikzpicture}[x=0.65cm,y=0.65cm]
\draw[-] (0,0)-- (0,-1);
\draw[-] (2,0)-- (2,-1);
\draw[-] (3,0).. controls(4,-0.5).. (5,0);
\draw[-] (4,0.5)-- (4,-1.5);
\draw[-] (3,-1).. controls(4,-0.5).. (5,-1);
\draw[-] (6,0)-- (6,-1);
\draw[-] (8,0)-- (8,-1);
\node at (-0.75,-0.5) {$e_0=$};
\node at (0,0.5) {\tiny $-r$};
\node at (1,-0.5) {\tiny $\cdots$};
\node at (2,0.5) {\tiny $-2$};
\node at (3,0.5) {\tiny $-1$};
\node at (5,0.5) {\tiny $1$};
\node at (6,0.5) {\tiny $2$};
\node at (7,-0.5) {\tiny $\cdots$};
\node at (8,0.5) {\tiny $r$};
\end{tikzpicture}
\caption{Generators of the Brauer algebra of type $C$, where $i=\{1, \cdots, r-1\}$.} \label{Generators of Brauer diagram of type C}
\end{figure}
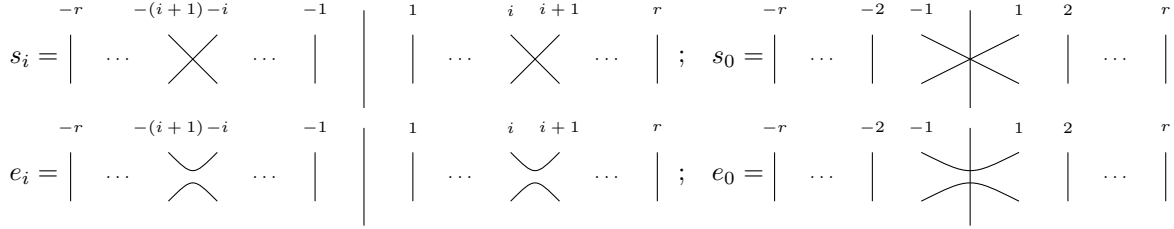
\end{center}

\subsection{Cellular structure} 
 The cellular structure of $\B(C_r, \delta)$ was studied in \cite{CLY}, following Graham and Lehrer \cite{GL}, whereas, our approach adopts the framework of Koenig and Xi in \cite{KX}. 
 
 Let $0\leq l \leq r$. An \textit{$(r,l)$-dangle} is a partition of $\{-r, \cdots, -1,1,\cdots, r\}$ into $2r-2l$ one-element subsets and $l$ two-element subsets. There is a natural action of the Dynkin diagram automorphism on the set of all $(r,l)$-dangles induced by permuting the nodes, which corresponds to reflection across the vertical axis of the diagram. Geometrically, each $(r,l)$-dangle can be viewed as a partial diagram with $2r$ vertices, consisting of $l$ horizontal edges and $2r-2l$ free vertices. A $(r,l)$-dangle is called \textit{symmetric} if all horizontal edges are mirrored across the vertical axis. Let $V_l$ denote the $K$-vector space spanned by all symmetric $(r,l)$-dangles. 
There exists a natural anti-automorphism $i$ on $\B(C_r, \delta)$ with $i^2=\mathrm{id}$, defined by reflecting diagrams across the horizontal axis that bisects the diagram into upper and lower halves. Thus, the involution $i$ on $\B(C_r,\delta)$ is given by $i(e \otimes f \otimes \bosig)= (f \otimes e \otimes \bosig^{-1})$, for $e, f \in V_l$, $\bosig \in \W_{r-2l}.$

\begin{proposition}\label{iterated inflation of Brauer algebras of type C for type C}[\cite{Bo}, Theorem 3.6]
    The Brauer algebra of type $C$ can be expressed as an iterated inflation of the group algebra of hyperoctahedral groups 
    $$ \B(C_r, \delta)= \bigoplus_{l=0}^{r} V_l \otimes V_l \otimes \W_{r-l},$$
    where  $l \in \{0, \cdots, r\}$.
\end{proposition}
The algebra $\W_{r-l}$ is cellular by [\cite{GL}, Theorem 5.5]. By Proposition \ref{iterated inflation of Brauer algebras of type C for type C} and [\cite{KX3}, Proposition 3.5], it follows that $\B(C_r, \delta)$ is  cellular. For each $l \in\{0, \cdots, r\}$, define $J_l$ to be the two-sided ideal of $\B(C_r, \delta)$ generated by all diagrams having at least $l$ horizontal edges. Therefore, we have a filtration of $\B(C_r,\delta)$: $$ \B(C_r, \delta) =J_0 \supseteq J_1 \supseteq \cdots \supseteq J_r\supseteq 0,$$ where each subquotient is isomorphic to an inflation $V_l \otimes V_l \otimes \W_{r-l}$ of $\W_{r-l}$ along $V_l$, as follows from [\cite{Bo}, Lemma 3.3]. 

For $0 \leq l \leq r$ and $ \delta \neq 0$, the idempotent $e_l$ was given by Bowman in \cite{Bo}. This idempotent is depicted in Figure \ref{Idempotent of Type C when delta non-zero}.  The idempotent $e_0$ is defined to be identity element of $\W_{r}$ in $\B(C_r,\delta)$.
\begin{center}
\begin{figure}[H]
\begin{tikzpicture}[x=1cm,y=0.65cm]
\draw[-] (-2,0)-- (-2,-1);
\draw[-] (0,0)-- (0,-1);
\draw[-] (1,0).. controls(4,-0.5)..(7,0);
\draw[-] (3,0).. controls(4,-0.25)..(5,0);
\draw[-] (4,0.5)-- (4,-1.5);
\draw[-] (8,0)-- (8,-1);
\draw[-] (10,0)-- (10,-1);
\draw[-] (1,-1).. controls(4,-0.5)..(7,-1);
\draw[-] (3,-1).. controls(4,-0.75)..(5,-1);
\node at (-3,-0.5) {$e_l=\frac{1}{\delta^{l}}$};
\node at (-2,0.5) {\tiny $-r$};
\node at (-1,-0.5) {\tiny $\cdots$};
\node at (0,0.5) {\tiny $-(l+1)$};
\node at (1,0.5) {\tiny $-l$};
\node at (2,-0.5) {\tiny $\cdots$};
\node at (3,0.5) {\tiny $-1$};
\node at (5,0.5) {\tiny $1$};
\node at (6,-0.5) {\tiny $\cdots$};
\node at (7,0.5) {\tiny $l$};
\node at (8,0.5) {\tiny $l+1$};
\node at (9,-0.5) {\tiny $\cdots$};
\node at (10,0.5) {\tiny $r$};
\end{tikzpicture}
\caption{Idempotent of $\B(C_r, \delta)$ for $\delta \neq 0$.}\label{Idempotent of Type C when delta non-zero}
\end{figure}
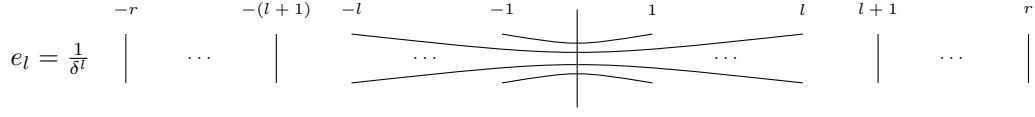
\end{center}

Finally, it follows from [\cite{Bo}, Theorem 3.6] that if $\delta \neq 0$, then $\B(C_r,\delta)$ is cellularly stratified (for more details, see [\cite{HHKP}, Definition 2.1]). 


\subsection{Permutation modules of $\B(C_r,\delta)$} 
In this subsection, our main aim is to define the permutation modules of $\B(C_r,\delta)$. To this end, we first set up the induction and restriction functors, followed by the split pair approach, as discussed in \cite{CGsplit}. From now on, we denote the algebra $\B(C_r, \delta)$ simply by $B$. For $0 \leq l \leq r$, we consider $\text{\textbf{mod}-}\W_{r-l}$ (resp. $\text{\textbf{mod}-}B$) to be the category of finitely generated right $\W_{r-l}$-modules (resp. $B$-modules). 
The exact split pair functors are defined as follows:
\begin{align*}
  \ind_l:\text{\textbf{mod}-}& \W_{r-l}  \longrightarrow  \text{\textbf{mod}-} B &\Res_l :\text{\textbf{mod}-} B \longrightarrow \text{\textbf{mod}-}\W_{r-l} \\
  & M \longmapsto  M \otimes_{\substack{\W_{r-l}}} e_l(B/J_{l+1})   &N \longmapsto  N\otimes_{B} Be_l.
\end{align*}
The functor $\ind_l$ operates within the $l$th layer, whereas $\Res_l$ retains all layers $m$, for $m \geq l$. We also define another induction functor that keeps all the layers $m \geq l$
\begin{align*}
  \Ind_l:\text{\textbf{mod}-} \W_{r-l}  &\longrightarrow  \text{\textbf{mod}-} B \\
  M &\longmapsto  M \otimes_{\substack{\W_{r-l}}} e_lB. 
\end{align*}
For each $0 \leq l \leq r$, the functor $\Res_l$ is right adjoint to $\Ind_l$. Furthermore, $\Res_l$ is a left inverse to $\ind_l$, however, in general, it is not a left inverse to $\Ind_l$.

Define the index set $\Lambda$ by  $\Lambda=\{(l, \blam): 0 \leq l \leq r, \blam \in \Lam_{r-l}\}$. For $(l, \blam),(m,\bmu)\in \Lambda$, we impose a partial order given by: 
\begin{equation*}
    (m,\bmu) \geq (l,\blam) \text{ if } m < l  (l=m \text{ and } \blam \trianglerighteq \bmu).
\end{equation*}

For $M \in \text{\textbf{mod}-} B$, define  its dual by $M^{*}=\Hom_{K}(M, K)$, where $K$ is the base field. The right action of an element $d \in B$ on  $f \in M^{*}$ is given by $(f\cdot d)(m)=f(m\cdot i(d))$, for all $m \in M$. This gives a right $B$-module structure on $M^{*}$. Since the functor $\ind_l$ is exact, the cell modules of $B$ can be obtained as the images of the cell modules of $\W_{r-l}$ under $\ind_l$. Accordingly 
 we denote the cell module of $B$ by $\ind_l S'(\blam)$, where $(l,\blam) \in \Lambda$ and $S'(\blam)$ is the dual Specht module of $\W_{r-l}$. Although one could alternatively define the cell module as $\ind_l S(\blam)$, where $S(\blam)$ is the Specht module of $\W_{r-l}$, we do not consider $\ind_l S(\bolam)$ in this paper.
%
The simple $B$-modules arise as quotients of the cell modules $\ind_l S'(\blam)$ for $(l,\blam) \in \Lambda$, as shown in  [\cite{GL}, Theorem 3.4]. Hence, the isomorphism classes of simple $B$-modules are indexed by pairs $(l, \blam) \in \Lambda$, where $\blam$ is a $p$-regular bi-partition of $r-l$ and $0 \leq l \leq r$.

We now define our main object - the permutation module of $B$ using the induction functor. For each $(l, \blam) \in \Lambda$, the \textit{permutation module} of $B$ is defined by 
\begin{equation*}
    M(l,\blam) := \Ind_l M(\blam) = M(\blam) \otimes_{\W_{r-l}} e_l B.
\end{equation*}
The dimension of $M(l,\blam)$ is independent of the parameter $\delta$ and the characteristic of the field $K$.

\section{Construction of Young module} \label{Construction of Young module for type C}

In this section, our main aim is to define the Young module as a unique direct summand of the permutation module, following an approach similar to that used for Brauer algebras in \cite{HP} and walled Brauer algebras in \cite{CGwalled}. The Young module of $B$ is constructed by extending the Young module of $\W_r$ as outlined in Theorem \ref{Decomposition of the permutation module of hyperoctahedral group}. 

\begin{theorem}\label{existence and uniqueness of Young module of type C}
For $(l,\blam) \in \Lambda$, there exists a unique direct summand of $M(l,\blam)$ with a quotient isomorphic to $\ind_l Y(\blam)$.
\end{theorem}
\begin{proof}
By Theorem \ref{Decomposition of the permutation module of hyperoctahedral group}, the permutation module $M(\blam)$ of $\W_r$ decomposes into a direct sum of  indecomposable Young modules, 
  \[
  M(\blam) = Y(\blam) \oplus \big(\underset{\bmu}{\oplus} Y(\bmu)^{ c_{\bmu}}\big).
  \]
 Applying the induction functor $\Ind_l$, we obtain a decomposition of the permutation module:
  \begin{align}
     M(l,\blam)= \Ind_lM(\blam)= \Ind_l Y(\blam)\oplus \big( \bigoplus_{(l,\bomu)\in \Lambda} \Ind_l Y(\bmu)^{ c_{\bomu}} \big).\nonumber\noindent
  \end{align}
  Let $\epsilon_i \in \End_{B} (\Ind_l Y(\blam)) $ be primitive idempotents such that $\underset{i=1}{\overset{s}{\sum}} \epsilon_{i} = 1_{ \End_{B} (\Ind_l Y(\blam))}$. Further decomposing $\Ind_l Y(\blam)$, we obtain
  \begin{equation}
  \Ind_l M(\blam)= \bigoplus_{i=1}^s (\Ind_l Y(\blam))\epsilon_i \oplus \big( \bigoplus_{(l,\bomu)\in \Lambda}  \Ind_l Y(\bmu)^{ c_{\bmu}} \big). \label{decomposition of Ind Y_blam for type C}
  \end{equation}

\textbf{Claim 1:} There exists a direct summand of $\Ind_l Y(\blam)$ with a quotient isomorphic to $\ind_l Y(\blam)$.\\

To prove this, it suffices to construct a surjective map $(\Ind_l Y(\blam))\epsilon_i \twoheadrightarrow \ind_l Y(\blam)$, for some $i$. Since $\Res_l$ is right adjoint to $\Ind_l$, and a left inverse of $\ind_l$, we have the following isomorphism:
\begin{align*}
\Hom_B \big( \Ind_l Y(\blam), \ind_l Y(\blam) \big ) 
& \cong \Hom_{\W_{r-l}} \big( Y(\blam), \Res_l \ind_l Y(\blam) \big )\\
& \cong \Hom_{\W_{r-l}} \big( Y(\blam),Y(\blam) \big )\\ 
&= \End_{\W_{r-l}} \big( Y(\blam) \big).
\end{align*}
This yields the following commutative diagram:
\begin{equation}\label{commutative diagram for existence of Young module for type C}
  \begin{tikzcd}[x=2cm,y=2cm]
 Y(\blam) \ar[r,"\phi"]\ar[d," \Ind_l"]&  Y(\blam)\ar[d, "\ind_l"]\\
\Ind_lY(\blam) \ar[r,"\Phi"] & \ind_lY(\blam).
\end{tikzcd}
\end{equation}
Let $\Phi \in \Hom_B \big( \Ind_l Y(\blam), \ind_l Y(\blam) \big ) $ be such that $\Phi \big(\Ind_l ( y) \big)= \ind_l \big( \phi(y)\big)$, for all $y \in Y(\blam)$. By choosing $\phi$ to be the identity map, it follows that $\Phi$ is surjective. Hence, from (\ref{decomposition of Ind Y_blam for type C}), there exists a direct summand of $\Ind_l M(\blam)$ (namely $(\Ind_l Y(\blam) )\epsilon_i$) that surjects onto $\ind_l Y(\blam)$ for some $i$.

\textbf{Claim 2:} For the same $i$ as above, $(\Ind_l Y(\blam)) \epsilon_i$ is the only direct summand of $\Ind_l Y(\blam)$ with a quotient isomorphic to $\ind_lY(\blam)$. 

Suppose there exists another direct summand  $(\Ind_l Y(\blam))\epsilon_j$ for $j \neq i$, with a quotient isomorphic to $\ind_l Y(\blam)$. Consider a surjective homomorphism $$\Psi: \Ind_l Y(\blam) \twoheadrightarrow \ind_l Y(\blam),$$ defined by
\begin{align*}
\Psi \big( (\Ind_l Y(\blam))\epsilon_j \big) &= \begin{cases} \ind_l Y(\blam) & ;\text{ if } j \neq i \\
0 &; \text{ otherwise. }
\end{cases}
\end{align*}
Then $\Psi$ sends an element of the form $y \otimes e_l$ to $ \phi(y) \otimes e_l+ \phi(y)\otimes e_l J_{l+1}$, for some non-zero $y \in Y(\blam)$ and $\phi \in \End_{\W_{r-l}} (Y(\blam))$. Since $\Psi$ is surjective, there exists $v= \sum\limits_{i} y_i \otimes e_l d_i$ (with $d_i \in B$) such that $\Psi(v)= \phi(y)\otimes (e_l+e_lJ_{l+1})$. Now, we compute 
\begin{equation*}
ve_l= \big( \sum_i^s y_i \otimes e_l d_i \big)e_l = \sum_i^s y_i \otimes e_l d_i e_l = y_1 \otimes e_l +\text{ lower terms. }
\end{equation*}
Evaluating $\Psi$ on an element of the form $ve_l$, we have 
\begin{equation}\label{only summand Young module 1 for type C}
\Psi(ve_l) = \Psi(y_1\otimes e_l+ \text{ lower terms })=  \phi(y_1) \otimes e_l + \text{ lower terms. }  
\end{equation}
Since $\Psi $ is a $B$-module homomorphism, we also have 
\begin{equation} \label{only summand Young module 2 for type C}
\Psi(ve_l)=  \Psi(v)e_l= \Psi\big(( 1\otimes e_l)+ \text{ lower term } \big)e_l= y \otimes e_l+ \text{ lower terms. }
\end{equation}

Comparing (\ref{only summand Young module 1 for type C}) and (\ref{only summand Young module 2 for type C}), we obtain $\phi(y_1)= y$. If $y_1=0$ then $y=0$, contradicting the assumption that $y \neq 0$. Hence $y_1\neq0$.
By Claim 1, there exists $\Phi \in \Hom_B(\Ind_l Y(\blam), \ind_l Y(\blam))$ such that $\Phi(ve_l)=y \otimes e_l+\text{ lower terms }\neq 0$. But $ve_l \in (\Ind_l Y(\blam))\epsilon_j$, which is a contradiction. Hence, the claim follows. 

\textbf{Claim 3:} There is no other summand of $\Ind_l Y(\bmu)$ with a quotient isomorphic to $\ind_l Y(\blam)$, for $\bmu \triangleright \blam$.

Assume that there exists a direct summand $\Ind_l Y(\bmu)$ of $\Ind_l M(\blam)$ with a quotient isomorphic to $\ind_l Y(\blam)$, for some $\bmu \triangleright \blam$. Then there exists a surjective $B$-module homomorphism $\Phi': \Ind_l Y(\bmu) \longrightarrow\ind_l Y(\blam)$. By an argument analogous to that used in (\ref{commutative diagram for existence of Young module for type C}), we obtain the following isomorphism:
\begin{equation*}
\Hom_B \big (\Ind_l Y(\bmu), \ind_l Y(\blam) \big) \cong \Hom_{\W_{r-l}} \big( Y(\bmu), Y(\blam) \big ).
\end{equation*}
Thus, $\Phi'$ is surjective only if $\phi' \in \Hom_{\W_{r-l}} \big( Y(\bmu), Y(\blam) \big )$ is also surjective. We can extend the epimorphism $\phi'$ to $\tilde{\phi'}: M(\bmu) \longrightarrow Y(\blam) $ as follows: for $m \in M(\bmu)$,
\begin{align*}
 \tilde{\phi'}(m) = \begin{cases} \phi'(m) &; \text{ if } m \in Y(\bmu)\\
0 &; \text{ otherwise}.
\end{cases}
\end{align*}
Let $\tilde{\tab}$ be a $\bomu$-bi-tableau. Consider $k_{\tab}$ to be the alternating column sum corresponding to the $\blam$-bi-tableau $\tab$. If $\blam \ntrianglerighteq \bmu$ and $|\lambda^{1}|>|\mu^{1}|$, then $\{\tilde{\tab}\}\cdot  k_{\tab}=0$ by [\cite{CanHyper}, Lemma 2.2]. Since the permutation module $M(\bmu)$ is generated by all $\bmu$-tabloids, we have $M(\bmu) k_{\tab}=0$ provided $\bmu \triangleright \blam$. Then all direct summands of $M(\bmu)$ act trivially under $k_{\tab}$, and in particular $Y(\bmu) k_{\tab}=0$. Since $Y(\bolam)k_{\tab}$ is the image of $Y(\bomu)k_{\tab}$ under the map $\phi'$, we conclude $Y(\bolam) k_{\tab}=0$.
%
%
%
This contradicts the fact that $\{\tab\}k_{\tab}$ is a non-zero element, as it generates the Specht module $S(\blam)$. Hence, the assumption leads to  a contradiction.
\end{proof}
The above construction allows us to define the Young modules of $B$. For $(l, \blam) \in \Lambda$, we denote by $Y(l,\blam)$ the unique direct summand of $M(l,\blam)$ that admits a surjection onto $\ind_l Y(\blam)$. We refer to this as the \textit{Young module} of $B$. 

\begin{proposition}\label{Characterization of Young module in type C}
The Young modules satisfy the following properties: let $(l, \blam),(m, \bmu) \in \Lambda$,
\begin{itemize}
\item[1.] The Young module $Y(l,\blam)$ appears exactly once in the decomposition of $M(l,\bolam)$.
\item[2.] If $l<m$, then $Y(l, \blam)$ does not appear as a direct summand of $M(m,\bmu)$.
\item[3.] $Y(l,\blam)$ appears as a direct summand of $M(l,\bmu)$ if and only if  $Y(\blam)$ is a direct summand of $M(\bmu)$, which occurs if $\blam \unrhd \bmu$.
\item[4.] If $Y(l,\blam)\cong Y(m,\bmu)$, then $(l,\blam)=(m,\bmu)$.
\end{itemize}
\end{proposition}
\begin{proof}
    \begin{itemize}
        \item[1.] This follows directly from Theorem \ref{existence and uniqueness of Young module of type C}.
        \item[2.] Suppose that $Y(l,\blam)$ appears as a direct summand of $M(m,\bmu)$ for $l<m$. Then there exists a surjective map $M(m,\bmu) \longrightarrow \ind_l Y(\blam)$. Since $\Res_m$ is right adjoint to $\Ind_m$, we have
        \begin{align*}
             \Hom_B(\Ind_m M(\bmu), \ind_l Y(\blam))&= \Hom_{\W_{r-l}} (M(\bmu), \Res_m\ind_l Y(\blam))\\
             &= \Hom_{W_{r-l}} (M(\bmu), Y(\blam) \otimes_{\W_{r-l}} e_l (B/ J_{l+1}) e_m).                
        \end{align*}
         However, if $l<m$, then $e_l (B/ J_{l+1} )e_m=0$, which implies $Y(l,\blam)$ can not appear as a direct summand of $M(m, \bmu)$. 
        \item[3.] Assume that $Y(l,\blam)$ appears as a direct summand of $M(l,\bmu)$. Then by (\ref{commutative diagram for existence of Young module for type C}), it follows that $Y(\blam)$ must also be a direct summand of $M(\bmu)$. By Theorem \ref{Decomposition of the permutation module of hyperoctahedral group}, this is possible only if $\blam \unrhd \bmu$.
        
        Conversely, if $Y(\blam)$ is a direct summand of $M(\bmu)$, then $\Ind_l Y(\blam)$ appears as a direct summand of $\Ind_l M(\bmu)$. Since $Y(l,\blam)$ is the unique direct summand of $\Ind_l Y(\blam)$ with quotient isomorphic to $\ind_l Y(\blam)$, it follows that $Y(l,\blam)$ must appears as a direct summand of $M(l,\bmu)$.
        \item[4.] Suppose that $Y(l,\blam)$ is isomorphic to $Y(m,\bmu)$. Since $Y(m,\bmu)$ is a direct summand of $M(m,\bmu)$, the same holds for $Y(l,\blam)$. By part (2) of Proposition \ref{Characterization of Young module in type C}, we have $m \leq l$. If $m=l$, then by Proposition \ref{Characterization of Young module in type C} (3), we have $\blam \unrhd \bmu$. Hence, we obtain $(m,\bmu) \geq (l, \blam)$. Similarly, by considering $Y(m,\bmu)$ as a direct summand of $M(l, \blam)$, one can show that $(l,\blam) \geq (m,\bmu)$. Therefore, we conclude $(m,\bmu)=(l,\blam)$. 
    \end{itemize}
\end{proof}
\begin{remark}
     The results of this section remain valid when $K$ is any commutative unital ring with an invertible element $\delta$. However, starting from the next section, we will require $K$ to be an algebraically closed field of characteristic $p$. Therefore, for clarity and consistency, we will henceforth work with this fixed choice of $K$.
\end{remark}


\section{Cell filtration multiplicity} \label{Cell filtration multiplicity for type C}

In this section, we use the cohomological properties of the group algebras of the hyperoctahedral groups, as discussed in Section \ref{Vanishing cohomology on hyperoctahedral group for hyperoctahedral group}, to establish the analogues of the results of Hemmer and Nakano from \cite{HN} in the context of Brauer algebra of type $C$. A $B$-module $M$ is said to admit a \textit{cell filtration} if there exists a sequence of submodules of $M$ $$M=M_0 \supseteq M_1 \supseteq \cdots \supseteq M_k \supseteq M_{k+1} =0,$$ such that each subquotient is isomorphic to a cell module of $B$. Let $\Theta:= \{\Theta(l,\blam)= \ind_l S'(\blam):(l,\blam)\in \Lambda\}$ denote the collection of all cell modules of $B$. Let $\F_B(\Theta)$ denote the category of right $B$-modules admitting a cell filtration. 

\subsection{Stratifying system}
The stratifying system was first introduced by Erdmann and Sanez in \cite{ES}. In this section, we aim to show that the cell modules of $B$ form a stratifying system. As a direct consequence, the cell filtration multiplicities are well-defined. We begin by recalling the following definition from [\cite{Er05}, Definition 2.1] and [\cite{ES}, Definition 1.1]. 

\begin{definition}\label{definition of stratifying system}
   Let $R$ be a finite-dimensional algebra over an algebraically closed field. Given a set of $R$-modules $\{ \Theta(i): i =1,  \cdots, r \}$ and a set of indecomposable $R$-modules $\{Y(i): i= 1, \cdots, r\}$, we say that $\{\big(\Theta(i), Y(i)\big):i= 1,\cdots,r\}$ is a \textit{stratifying system} if it satisfies the following properties:
\begin{itemize}
    \item[(i)] $\Hom_R \big(\Theta(i+k), \Theta(i) \big)=0 $ for $k \geq 1$ and for all $ i\leq r$.
    \item[(ii)] For each $i$, there is a short exact sequence 
    \[
    \begin{tikzcd}
        0 \arrow[r] & \Theta(i) \arrow[r] & Y(i) \arrow[r] & Z(i) \arrow[r] &0,
    \end{tikzcd}
    \]
    and $Z(i)$ is filtered by $\Theta(j)$ with $j <i$.
    \item[(iii)] $\Ext_R^1 (\F(\Theta), Y)=0$, where $Y= \bigoplus\limits_{i=1}^r Y(i)$ and $\F(\Theta)$ denotes the category of $R$-modules admitting a $\Theta$-filtration. 
\end{itemize}
\end{definition}
To establish the existence of stratifying system, it is enough to prove the following result:
\begin{theorem}[\cite{ES}, Theorem 1.8] \label{existence of stratifying system}
Assume we have indecomposable $R$-modules $\{\Theta(i): i =1, \cdots, r\}$ such that 
\begin{itemize}
    \item[1.] $\Hom_R (\Theta(j), \Theta(i)) =0$ for $j >i$.
    \item[2.] $\Ext_R^1 (\Theta(j), \Theta(i)) =0$ for $j \geq i$.
\end{itemize}
Then there exists modules $\{Y(i): i=1,\cdots, r\}$ such that $\{(\Theta(i), Y(i))\}_{i=1}^r$ is a stratifying system. 
\end{theorem}
We now proceed to verify that the cell modules of $B$ satisfy the conditions stated in Theorem \ref{existence of stratifying system}.
\begin{lemma}\label{Hom between two cell modules vanishing condition for type C}
    Let $\mathrm{char}~K \neq 2$, and $\delta \neq 0$. For $(l,\blam), (m,\bmu)\in \Lambda$, we have 
    \begin{equation*}
         \Hom_B(\ind_l S'(\blam), \ind_m S'(\bmu))=0, \text{ if }  (m, \bmu)> (l,\blam).
    \end{equation*}
\end{lemma}
\begin{proof}
 We have the following isomorphism
\begin{align}
    \Hom_B(\ind_l S'(\blam), \ind_m S'(\bmu)) &\cong \Hom_{B}( S'(\blam) \otimes_{\W_{r-l}} e_l (B/J_{l+1}),  \ind_m S'(\bmu) )& \nonumber \\
    & \cong \Hom_{\W_{r-l}} (S'(\blam), \Hom_B (e_l (B/J_{l+1}), \ind_m S'(\bmu)))& \hfill \nonumber\label{Hom between two indl for type C}\\
    & \cong \Hom_{\W_{r-l}} (S'(\blam), S'(\bmu) \otimes_{\W_{r-m}} e_m(B/J_{m+1})e_l).&
\end{align}
If $l>m$, then $e_m (B/J_{m+1})e_l=0$, and hence the $\Hom$-space vanishes. If $l=m$, we have
$e_l(B/J_{l+1})e_l \cong \W_{r-l}$ as a $\W_{r-l}$-module. Therefore, we get
\begin{equation*}
    \Hom_B(\ind_l S'(\blam), \ind_l S'(\bmu)) \cong \Hom_{\W_{r-l}}( S'(\blam), S'(\bmu)). 
\end{equation*}
Moreover, if $\mathrm{char}~K \neq 2$, and $\bmu \ntrianglerighteq \blam$, it follows from Proposition \ref{Hom between Specht module of hyperoctahedral group for type C} that the right-hand side of the above congruence vanishes. Hence, the proof follows.
\end{proof}
\begin{lemma}\label{Ext between two cell modules vanishing condition for type C}
     Let $\mathrm{char}~K \neq 2,3$, and $\delta \neq 0$. For $(l,\blam),(m,\bmu)\in \Lambda$, we have 
    \begin{equation*}
         \Ext_B^1(\ind_l S'(\blam), \ind_m S'(\bmu))=0 \text{ if } (m, \bmu)\geq (l,\blam).
    \end{equation*}
\end{lemma}
\begin{proof}
    Consider the following short exact sequence in $\text{\textbf{mod}-}B$ 
    \begin{equation}\label{Ext between two cell modules vanishing condition for type C exact sequence 1}
    \begin{tikzcd}
        0 \arrow[r] &\ker \phi \arrow[r] &e_l B \arrow[r, "\phi"] &\ind_l \W_{r-l} \arrow[r] &0,
    \end{tikzcd}    
    \end{equation}
    where $e_lB$ is a projective right $B$-module. The map $\phi$ sends the Brauer diagrams of type $C$ with at least $l$ horizontal edges to diagrams with exactly $l$ horizontal edges.  Since $e_l B$ is projective, $\phi$ is surjective,  and the sequence in (\ref{Ext between two cell modules vanishing condition for type C exact sequence 1}) is exact.

    Applying the contravariant functor  $\Hom_B(-, \ind_m S'(\bmu))$ to this sequence, we obtain the following exact sequence
    \begin{equation}\label{Ext between two cell modules vanishing condition for type C exact sequence 2}
    \begin{tikzcd}
        0 \arrow[r] &\Hom_B(\ind_l \W_{r-l},\ind_m S'(\bmu)) \arrow[r] &\Hom_B(e_lB,\ind_m S'(\bmu))        
        \ar[draw=none]{d}[name=X, anchor=center]{}
     \ar[rounded corners,
            to path={ -- ([xshift=2ex]\tikztostart.east)
                      |- (X.center) \tikztonodes
                      -| ([xshift=-2ex]\tikztotarget.west)
                      -- (\tikztotarget)}]{dll}[at end]{} \\       
        \Hom_B(\ker \phi,\ind_m S'(\bmu)) \arrow[r] & \Ext_B^1(\ind_l \W_{r-l},\ind_m S'(\bmu)) \arrow[r] & \Ext_B^1(e_lB,\ind_m S'(\bmu))=0.
    \end{tikzcd}    
    \end{equation}
    Since $e_lB$ is projective, it follows that $\Ext_B^1(e_lB, \ind_m S'(\bmu))=0$. Moreover, $\ker \phi$ consists of Brauer diagrams of type $C$ with at least $l+1$ horizontal edges that fix the top configuration containing the partial diagram of $e_l$. By the same argument as in (\ref{Hom between two indl for type C}), we get 
    \begin{equation*}
        \Hom_B(\ker \phi, \ind_m S'(\bmu))=0 \text{ if } m<l+1.
    \end{equation*}
    Hence, from (\ref{Ext between two cell modules vanishing condition for type C exact sequence 2}), we conclude that 
    \[
    \Ext_B^1(\ind_l \W_{r-l},\ind_m S'(\bmu))=0 \text{ if }m>l.
    \]
    
      It remains to show that when $m=l$, and  $\bmu \ntriangleright \blam$, we have $\Ext_B^1(\ind_l S'(\blam),\ind_l S'(\bmu))=0$. Consider the following short exact sequence 
    \begin{equation}\label{Ext between two cell modules vanishing condition for type C exact sequence 3}
    \begin{tikzcd}
        0 \arrow[r] &\ind_l\ker \varphi \arrow[r] &\ind_l P \arrow[r,] &\ind_l S'(\blam) \arrow[r] &0,
    \end{tikzcd}    
    \end{equation}
    where $P$ is the projective cover of $S'(\blam)$, and $\ker \varphi$ is the kernel of the surjection $P \xlongrightarrow{\varphi} S'(\blam)$. The exactness of (\ref{Ext between two cell modules vanishing condition for type C exact sequence 3}) follows from the exactness of the functor $\ind_l$. Applying the contravariant functor $\Hom_B(-, \ind_l S'(\bmu))$, we obtain the long exact sequence: 
    \begin{equation} \label{Ext between two cell modules vanishing condition for type C exact sequence 6}
    \begin{tikzcd}
        0 \arrow[r] &\Hom_B(\ind_l S'(\blam),\ind_l S'(\bmu)) \arrow[r] &\Hom_B(\ind_l P,\ind_l S'(\bmu))        
        \ar[draw=none]{d}[name=X, anchor=center]{}
     \ar[rounded corners,
            to path={ -- ([xshift=2ex]\tikztostart.east)
                      |- (X.center) \tikztonodes
                      -| ([xshift=-2ex]\tikztotarget.west)
                      -- (\tikztotarget)}]{dll}[at end]{} \\       
        \Hom_B(\ind_l\ker \varphi,\ind_l S'(\bmu)) \arrow[r] & \Ext_B^1(\ind_l S'(\blam),\ind_l S'(\bmu)) \arrow[r] & \Ext_B^1(\ind_l P,\ind_l S'(\bmu))=0.
    \end{tikzcd}    
    \end{equation}
        
    Since $\ind_l$ preserves projectivity, $\ind_l P$ is projective, so $\Ext_B^1(\ind_l P,\ind_l S'(\bmu))=0$. Thus, to conclude $\Ext_B^1(\ind_l S'(\blam),\ind_l S'(\bmu))=0$ for $\blam \unrhd \bmu$, it suffices to show $$\Hom_B(\ind_l\ker \varphi,\ind_l S'(\bmu))=0.$$

   We now compute the following term from (\ref{Ext between two cell modules vanishing condition for type C exact sequence 6})
    \begin{align*}
        \Hom_B\big(\ind_l \ker\varphi, \ind_l S'(\bmu)\big)&\cong \Hom_{\W_{r-l}} (\ker \varphi, \Hom_B(e_l(B/J_{l+1}), S'(\bmu) \otimes_{\W_{r-l}} e_l(B/J_{l+1})\big)\\
        &\cong  \Hom_{\W_{r-l}} (\ker \varphi,S'(\bmu) \otimes_{\W_{r-l}} e_l(B/J_{l+1})e_l)\\
        &\cong  \Hom_{\W_{r-l}} (\ker \varphi,S'(\bmu)).
    \end{align*}
    Thus, it remains to show $\Hom_{\W_{r-l}}(\ker \varphi, S'(\bmu))=0$.
    
    Applying the exact functor $\Res_l$ to (\ref{Ext between two cell modules vanishing condition for type C exact sequence 3}), and then applying $\Hom_{\W_{r-l}}(-, S'(\bmu))$, we get    
    \begin{equation*} \label{Ext between two cell modules vanishing condition for type C exact sequence 5}
    \begin{tikzcd}
        0 \arrow[r] &\Hom_{\W_{r-l}}( S'(\blam),S'(\bmu)) \arrow[r] &\Hom_{\W_{r-l}}(P, S'(\bmu))        
        \ar[draw=none]{d}[name=X, anchor=center]{}
     \ar[rounded corners,
            to path={ -- ([xshift=2ex]\tikztostart.east)
                      |- (X.center) \tikztonodes
                      -| ([xshift=-2ex]\tikztotarget.west)
                      -- (\tikztotarget)}]{dll}[at end]{} \\       
        \Hom_{\W_{r-l}}(\ker  \varphi, S'(\bmu)) \arrow[r] & \Ext_{\W_{r-l}}^1( S'(\blam), S'(\bmu)) \arrow[r] & \cdots.
    \end{tikzcd}    
    \end{equation*}
     If $\mathrm{char}~K \neq 2,3$ and $\bmu \ntriangleright \blam$, then by Proposition \ref{Ext between Specht module of hyperoctahedral group for type C},  $\Ext_{\W_{r-l}}^1( S'(\blam), S'(\bmu))=0$. Moreover, since $P$ is the projective cover of $S'(\blam)$, $\Hom_{\W_{r-l}}(S'(\blam), S'(\bmu))=0$ when  $\ch~K\neq 2$ and $ \bmu \ntrianglerighteq \blam$. Therefore, $\Hom_{\W_{r-l}}(P,S'(\bmu))=0$, which implies that $\Hom_{\W_{r-l}}(\ker \varphi, S'(\bmu))=0$, completing the proof.
\end{proof}
\begin{proposition}\label{stratifying system of brauer algebra of type C}
Let $\mathrm{char}~K \neq 2,3$, and $\delta \neq 0$. Then there exists a stratifying system of $B$.
\end{proposition}
\begin{proof}
The cell modules for $B$ are given by the set $\{\ind_l S'(\blam): 0 \leq l \leq r, \blam \in \Lam_{r-l}\}$. From (\ref{Hom between two indl for type C}), we have
\begin{align*}
    \Hom_B(\ind_l S'(\blam), \ind_l S'(\bmu)) 
    & \cong \Hom_{\W_{r-l}} (S'(\blam), S'(\bmu)  ).
\end{align*}
In particular, this implies $\End_B (\ind_l S(\blam)) \cong \End_{\W_{r-l}} (S(\blam))$. Since each endomorphism ring on the right hand side is local, it follows that all cell modules indexed by elements of $\Lambda$ are indecomposable. It remains to verify that the set of cell modules of $B$ satisfies conditions (1) and (2) of Theorem \ref{existence of stratifying system}. If $\mathrm{char}~K \neq 2,3$ and $\delta \neq 0$, the conditions (1) and (2) of the Theorem \ref{existence of stratifying system} are satisfied by Lemma \ref{Hom between two cell modules vanishing condition for type C} and Lemma \ref{Ext between two cell modules vanishing condition for type C}, respectively. Hence, the proof follows.
\end{proof}
\begin{corollary}\label{cell-filtration multiplicity are well-dependent for type C}
Let $\mathrm{char}~K \neq 2,3$, and $\delta \neq 0$. If a $B$-module $M$ admits a cell filtration, then the filtration multiplicities are independent of the choice of filtration. 
\end{corollary}
\begin{proof}
The proof is an immediate consequence of Proposition \ref{stratifying system of brauer algebra of type C} and [\cite{ES}, Lemma 1.4].
\end{proof}

\subsection{Stabilizer of a partial diagram}\label{stabilizer of a partial diagram for type C}

In this subsection, our goal is to determine the $\W_{l}$-stabilizer of a partial diagram $v$ in $V_l$ which contains the partial diagram of $e_l$. More precisely, we aim to identify $\Stab_{\W_{l}} (v)$, where the partial diagram $v \in V_l$ is described in  Figure \ref{Partial diagram having exactly l horizontal edges for type C}.
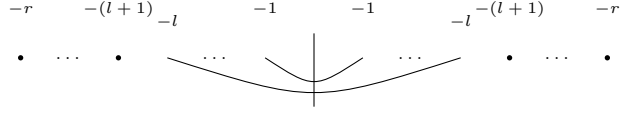
\begin{figure}[H]
\begin{tikzpicture}[x=0.65cm,y=0.65cm]
\draw[-] (1,0).. controls(4,-0.95)..(7,0);
 \draw[-] (3,0).. controls(4,-0.650)..(5,0);
\draw[-] (4,0.5)--(4,-1);
\fill (0,0) circle[radius = 1pt];
\fill (-2,0) circle[radius = 1pt];
\fill (8,0) circle[radius = 1pt];
\fill (10,0) circle[radius = 1pt];
\node at (-2,1) {\tiny $ -r$};
\node at (0,1) {\tiny $-(l+1)$};
\node at (3,1) {\tiny $-1$};
\node at (1,0.75) {\tiny $-l$};
\node at (10,1) {\tiny $ -r$};
\node at (8,1) {\tiny $-(l+1)$};
\node at (5,1) {\tiny $-1$};
\node at (7,0.75) {\tiny $-l$};
\node at (-1,0) {\tiny $\cdots$};
\node at (2,0) {\tiny $\cdots$};
\node at (6,0) {\tiny $\cdots$};
\node at (9,0) {\tiny $\cdots$};
\end{tikzpicture}
\caption{Partial diagram having exactly $l$ horizontal edges.}\label{Partial diagram having exactly l horizontal edges for type C}
\end{figure} 

The algebra $\W_{l}$ acts on $v$ by permuting its horizontal edges. A partial diagram $v$ remains invariant under the following operations: 
\begin{enumerate}
    \item[1.] interchanging the starting and ending points of a horizontal edge, 
    \item[2.] swapping one horizontal edge with another, 
    \item[3.] performing both operations simultaneously.
\end{enumerate}
Interchanging the initial and terminal vertices of each horizontal edge contributes $l$ copies of $\Z/2\Z$, while swapping one horizontal edge to another horizontal edge corresponds to symmetric group $\Sg_{l}$. These symmetries together generate $\W_l=(\mathbb{Z}_2)^{l}\rtimes \Sg_l$, all of them stabilize the partial diagram $v$. Thus, the stabilizer of $v$ under the action of $\W_l$ is the entire group algebra $\W_l$, and we define the stabilizer subalgebra $\h_l$ by %
\[
\h_{l}:=\Stab_{\W_{l}}(v)= \W_{l}.
\]

\subsection{Cell filtration}
In this section, we prove that both the permutation module and the Young module of $B$ admit a cell filtration. 

\begin{theorem}\label{cell filtration of permutation module for type C}
Let $\mathrm{char}~K \neq 2,3$, and $\delta \neq 0$. For $(l,\blam) \in \Lambda$, the permutation module $M(l,\blam)$ has a cell filtration. 
\end{theorem}
\begin{proof}
The algebra $B$ has a filtration by two-sided ideals $B = J_0 \supseteq J_{1} \supseteq \cdots \supseteq J_{r-1} \supseteq J_{r} \supseteq 0$. For each $0 \leq m \leq r$, this induces a short exact sequence of $B$-$B$ bimodules 
\[
\begin{tikzcd}
    0 \arrow[r] & J_{m+1} \arrow[r] & J_m \arrow[r] &J_m/J_{m+1} \arrow[r] & 0.
\end{tikzcd}
\]
 Since the functor $e_l B\otimes_B -$ is exact, we get the following short exact sequence: 
\[
\begin{tikzcd}
    0 \arrow[r] & e_lJ_{m+1} \arrow[r] & e_lJ_m \arrow[r] & e_l(J_m/J_{m+1}) \arrow[r] & 0.
\end{tikzcd}
\]
The above sequence is split exact as a left $\W_{r-l}$-module, due to the embedding $e_l(J_m/J_{m+1}) \hookrightarrow e_lJ_m$. Applying the exact functor $M(\blam) \otimes_{ \W_{r-l}} -$ to the above split exact sequence yields a short exact sequence of $\W_{r-l}$-$B$-bimodules:
\[
\begin{tikzcd}
0 \arrow[r] & M(\blam)\otimes_{\W_{r-l}} e_lJ_{m+1} \arrow[r] & M(\blam) \otimes_{\W_{r-l}} e_lJ_{m}   \arrow[r] & M(\blam)\otimes_{\W_{r-l}} e_l(J_{m}/ J_{m+1})  \arrow[r] &0.
\end{tikzcd} 
\]
This gives a filtration of $M(l,\blam) =M(\blam)\otimes_{ \W_{r-l}} e_lB= \Ind_{l} M(\blam)$ as:
\begin{align*}
M(\blam) \otimes_{ \W_{r-l}} e_lB \supseteq M(\blam) \otimes_{\W_{r-l}} e_lJ_1 \supseteq \cdots \supseteq  M(\blam) \otimes_{\W_{r-l}} e_lJ_{r-1}  \supseteq M(\blam) \otimes_{\W_{r-l}} e_lJ_r \supseteq 0,
\end{align*}
with subquotients $M^{m}(l,\blam):= M(\blam) \otimes_{\W_{r-l}} e_l(J_{m}/ J_{m+1})$, for $l \leq m\leq r$. If $m<l$, then $e_l(J_m/J_{m+1})=0$, so we only need to show that each subquotient $M^{m}(l,\blam)$ has a cell filtration.

The permutation module $M(l,\blam)$ consists of elements of the form $m\otimes_{\W_{r-l}} e_ld$, where $m \in M(\blam), d \in B$, and $e_ld \in B$ has at least $l$ horizontal edges containing the partial diagram of $e_l$ in its top. Similarly, for each $m \geq l$, the $K$-space $M^m(l,\blam)$ has a basis consisting of elements of the form $m \otimes_{\W_{r-l}} e_ld$, where $m \in M(\blam)$, $d\in J_m/J_{m+1}$ and $e_ld$ has exactly $m$ horizontal edges containing the partial diagram of $e_l$ in its top. The remaining $m-l$ symmetric horizontal edges in the top configuration of $e_ld$ are connected outside the set of vertices $\{-l,-(l-1), \cdots, -1,1, \cdots, l-1,l\}$. 
Since the elements of $M^m(l,\blam)$ lie in the tensor product over $\W_{r-l}$, for any $d \in B$, there exists $\bosig \in \W_{r-l}$ such that $\bosig$ permutes the additional $m-l$ horizontal edges into the partial diagram of $e_m$. Figure \ref{Any element of Wn-l will permute additional m-l horizontal edges} illustrates an example of this scenario.
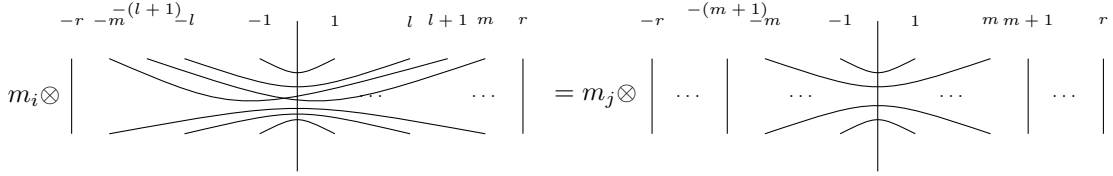
\begin{figure}[H]
\begin{tikzpicture}[x=0.5cm,y=1cm]
\draw[-] (-2,0)-- (-2,-1);
\draw[-] (0,0).. controls(4.5,-0.75).. (9,0);
\draw[-] (-1,0).. controls(2.5,-0.75).. (8,0);
\draw[-] (1,0).. controls(4,-0.50)..(7,0);
\draw[-] (3,0).. controls(4,-0.25)..(5,0);
\draw[-] (4,0.5)-- (4,-1.5);
\draw[-] (10,0)-- (10,-1);
\draw[-] (1,-1).. controls(4,-0.65)..(7,-1);
\draw[-] (3,-1).. controls(4,-0.75)..(5,-1);
\draw[-] (-1,-1).. controls(4,-0.55)..(9,-1);
\node at (-3,-0.5) {$m_i \otimes$};
\node at (-2,0.5) {\tiny $-r$};
\node at (0,0.65) {\tiny $-(l+1)$};
\node at (-1,0.5) {\tiny $-m$};
\node at (1,0.5) {\tiny $-l$};
\node at (3,0.5) {\tiny $-1$};
\node at (5,0.5) {\tiny $1$};
\node at (6,-0.5) {\tiny $\cdots$};
\node at (7,0.5) {\tiny $l$};
\node at (8,0.5) {\tiny $l+1$};
\node at (9,-0.5) {\tiny $\cdots$};
\node at (9,0.5) {\tiny $m$};
\node at (10,0.5) {\tiny $r$};
\end{tikzpicture}
\begin{tikzpicture}[x=0.5cm,y=1cm]
\draw[-] (-2,0)-- (-2,-1);
\draw[-] (0,0)-- (0,-1);
\draw[-] (1,0).. controls(4,-0.5)..(7,0);
\draw[-] (3,0).. controls(4,-0.25)..(5,0);
\draw[-] (4,0.5)-- (4,-1.5);
\draw[-] (8,0)-- (8,-1);
\draw[-] (10,0)-- (10,-1);
\draw[-] (1,-1).. controls(4,-0.5)..(7,-1);
\draw[-] (3,-1).. controls(4,-0.75)..(5,-1);
\node at (-3.5,-0.5) {$=m_j \otimes $};
\node at (-2,0.5) {\tiny $-r$};
\node at (-1,-0.5) {\tiny $\cdots$};
\node at (0,0.65) {\tiny $-(m+1)$};
\node at (1,0.5) {\tiny $-m$};
\node at (2,-0.5) {\tiny $\cdots$};
\node at (3,0.5) {\tiny $-1$};
\node at (5,0.5) {\tiny $1$};
\node at (6,-0.5) {\tiny $\cdots$};
\node at (7,0.5) {\tiny $m$};
\node at (8,0.5) {\tiny $m+1$};
\node at (9,-0.5) {\tiny $\cdots$};
\node at (10,0.5) {\tiny $r$};
\end{tikzpicture}
\caption{Any element of $\W_{r-l}$ will permute additional $m-l$ horizontal edges. }\label{Any element of Wn-l will permute additional m-l horizontal edges}
\end{figure}

Thus, the basis elements of $M^{m}(l,\blam)$ are of the form
\[
m_i \otimes e_l d=m_i\bosig^{-1} \otimes \bosig e_ld= m_j \otimes e_m d,
\]
where $e_md \in e_l (J_m/J_{m+1})$ and $m_i$ and $m_j$ are basis elements of $M(\blam)$. Define the group algebra of stabilizer subgroup of $\W_{r-l}$ by $\h_{m-l}(e_md)=\{\pmb{\omega}\in \W_{r-l}:\pmb{\omega} \cdot e_md= e_m d\}$. In particular, the elements of $\h_{m-l}(e_md)$ only permute the vertices $\{-(r-m),\cdots, -(r-l+1),(r+l+1), \cdots, (r+m)\}$. By Section \ref{stabilizer of a partial diagram for type C}, we have $\h_{m-l}(e_md)= \W_{m-l}$. Since this is independent of the choice of $e_md$, we write $\h_{n-l}$ instead of $\h_{n-l}(e_md)$. Additionally, the equality $m_i \otimes e_md=m_j \otimes e_m d$ holds if there exists $\pmb{\omega} \in \h_{m-l}$ such that $m_i \pmb{\omega}=m_j$, since $m_i \otimes e_m d=m_i \otimes \pmb{\omega} \cdot e_md=m_i\cdot \pmb{\omega} \otimes e_md=m_j \otimes e_md $. Hence, the subquotient consists of elements of the form $\{m\otimes e_m d:m \in M_m , e_m d \in e_l (J_m/J_{m+1})\}$, where $M_m$ is the $\W_{r-l}$-module defined by
\[
M_m=\frac{M(\blam)}{\langle m-m'\pmb{\omega}:m,m' \in M(\blam), \pmb{\omega} \in \h_{m-l}\rangle}.
\]
The $\W_{r-l}$-module $M_m$ is the largest quotient of $M(\blam)$ on which $\h_{m-l}$ acts trivially. Thus, we have $M_m=\W_{\blam}\backslash \W_{r-l}/ \h_{m-l}.$ Each representative of $M_m$ corresponds to an orbit of $\h_{m-l}$ on $\W_{\blam}\backslash \W_{r-l}$. The induced action of $\W_{r-m}$ on $M_m$ permutes these basis elements, making $M_m$ a permutation module of $\W_{r-m}$. To show that $M_m$ is the Young permutation module of $\W_{r-m}$, it remains to prove that the stabilizer of $M_m$ is isomorphic to a Young subgroup of $\W_{r-m}$. Take any element $\bal \in \W_{r-l}$. Since $\W_{r-m}$ commutes with $\h_{m-l}$, we have
 \begin{align*}
    \Stab_{\W_{r-m}} (\W_{\blam} \bal \h_{m-l}) &= \{\bal' \in \W_{r-m}:(\W_{\blam} \bal \h_{m-l})\cdot \bal'= \W_{\blam} \bal \h_{m-l}\}\\
    &= \{\bal' \in \W_{r-m}:(\W_{\blam} (\bal\cdot \bal') \h_{m-l})= \W_{\blam} \bal \h_{m-l}\}\\
    &= \Stab_{\W_{r-m}} (\W_{\blam} \bal) \\
    &=\W_{r-m} \cap \bal^{-1} \W_{\blam} \bal.
 \end{align*}
 Let us denote the group algebra of the stabilizer subgroup $\W_{r-m} \cap \bal^{-1} \W_{\blam} \bal$ by $\W_{\banu}$ for some $\banu \in \Lam_{r-m}$. Hence, the stabilizer of the permutation module corresponding to the double coset representative $\bal\in \W_{\blam}\backslash \W_{r-m}/ \h_{m-l}$ is given by $\W_{\banu}$, which is a Young subgroup of $\W_{r-m}$.

 Therefore, we can express $M^{m}(l,\blam)\cong \big(\ind_m \underset{\banu}{\oplus} M(\banu)\big)$, for each $(m,\banu)\in \Lambda$. If $\ch~K \neq 2$, then by [\cite{Ma}, Corollary 5.11], the permutation module $M(\banu)$ has a dual Specht filtration. Moreover, the exact functor $\ind_m$ sends each dual Specht module to a cell module. Thus, $M^m(l,\blam)$ has a cell filtration, since $\ind_m S'(\banu)$ is a cell module for every $(m,\banu)\in \Lambda$. 
 \end{proof}

\begin{corollary}
    Let $\mathrm{char}~K \neq 2,3$, and $\delta \neq 0$. For $(l,\blam) \in \Lambda$, every direct summand of $M(l,\blam)$ admits has a cell filtration. In particular, the Young module $Y (l,\blam)$ admits a cell filtration. 
\end{corollary}
\begin{proof}
    By Theorem \ref{cell filtration of permutation module for type C}, the permutation module $M(l,\blam)$ has a cell filtration, and hence lies in the category $ \F_B(\Theta)$. It follows from [\cite{Do}, Proposition A2.2(vi)] that all the direct summands of $M(l,\blam)$ also belong to $\F_B(\Theta)$. Consequently,  $Y(l,\blam) \in \F_B(\Theta)$.
\end{proof}

\subsection{Exactness of the functor $\Hom_B(M(l,\blam), -)$}

In this section, we prove that the functor $\Hom$ is exact in the category of cell-filtered $B$-modules $\F_B(\Theta)$. We say a $\W_r$-module $N$ has a \textit{dual Specht} \textit{filtration} if there exists a chain of submodules
\[
N=N_1 \supseteq N_2 \supseteq \cdots \supseteq N_k \supseteq 0,
\]
such that each subquotient is isomorphic to a dual Specht module of $\W_r$. Let the category of finitely generated right $\W_r$-modules having a dual Specht filtration be denoted by  $\F_{\W_r}(S)$.

\begin{lemma}\label{LR rule}
Let $\ch~K \neq 2$ and $r=r_1+r_2$. If $\bolam=(\lambda^1,\lambda^2) \in \Lam_{r_1} $ and $\bmu =(\mu^1,\mu^2) \in \Lam_{r_2}$, then $$\Ind_{\W_{r_1} \times \W_{r_2}}^{\W_r} S'(\bolam) \otimes_K S'(\bmu) \cong \bigoplus_{\bonu=(\nu^1,\nu^2)\in \Lam_r} L_{\bolam, \bmu}^{\bonu} S'(\bonu),$$ where $L_{\bolam, \bmu}^{\bonu}=\prod_{i=1}^2 L_{\lambda^i,\mu^i}^{\nu^{i}}$; and for each $i$, $L_{\lambda^i,\mu^i}^{\nu^{i}}$ is the Littlewood-Richardson coefficient.
\end{lemma}
\begin{proof}
The proof proceeds similar to the argument [\cite{RY03}, Theorem 4.4].
\end{proof}

\begin{lemma}\label{Res l of a cell module has a dual Specht filtration for type C}
    Let $\mathrm{char}~K \neq 2,3$, and $\delta \neq 0$. For $(m,\banu) \in \Lambda$ with $m \geq l$, we have $\Res_l \ind_m S'(\banu) \in \F_{\W_{r-l}}(S)$.
\end{lemma}
    \begin{proof}
    Any cell module of $B$ is of the form $\ind_m S'(\bmu)$, for some $(m,\bmu) \in \Lambda$. Hence, $\Res_l \ind_m S'(\banu)= S'(\banu)\otimes_{\W_{r-m}} e_m(B/J_{m+1})e_l$. If $m < l$, then $e_m(B/J_{m+1})e_l=0$. Therefore, we assume $m \geq l$. The vector space $e_m(B/J_{m+1})e_l$ is spanned by the Brauer diagrams of type $C$ with exactly $m$ horizontal edges, such that the top and bottom configuration contain the partial diagram of $e_m$ and $e_l$, respectively. This space is a right $\W_{r-l}$ module isomorphic to the $\W_{r-l}$-permutation module $\W_{r-l}/\h_{m-l}$, where $\h_{m-l}$ is as in Subsection \ref{stabilizer of a partial diagram for type C}. For $(m,\banu) \in \Lambda$, we have 
   \begin{align*}
        \Res_l \ind_m S'(\banu)&= S'(\banu)\otimes_{\W_{r-m}} e_m(B/J_{m+1})e_l &\\
        &= \Ind_{\W_{r-m} \times \W_{m-l}}^{\W_{r-l}} \Big( \Ind_{\h_{m-l}}^{\W_{m-l}} 1_{\h_{m-l}} \otimes_K S'(\banu) \Big)&\\
        &= \Ind_{\W_{r-m} \times \W_{m-l}}^{\W_{r-l}} \Big(M\big((1^{(r-m)}),(m-l)\big) \otimes_K S'(\banu) \Big) &
    \end{align*}
   
    It follows from [\cite{Ma}, Corollary 5.11] that $ M\big((1^{(r-m)}),(m-l)\big)$ has a dual Specht filtration, provided $\ch~K\neq 2,3$. Finally, by applying the characteristic free version of Littlewood-Richardson rule given in Lemma \ref{LR rule}, $\Res_l \ind_m S'(\blam) \in \F_{\W_{r-l}}(S)$.
    \end{proof}   
\begin{corollary}\label{Res l of any module having cell filtration has dual Specht filtration in type C}
    Let $\ch~K \neq 2,3$, and $\delta\neq 0$. If $X \in \F_{\W_{r-m}}(S)$, then $\Res_l \ind_m X \in \F_{\W_{r-l}}(S)$ for $0 \leq l \leq m \leq r$. In particular, if $Y \in \F_B(\Theta)$ then $\Res_l Y \in \F_{\W_{r-l}}(S).$
\end{corollary}
\begin{proof}
     Since the functor $\ind_m$ is exact and sends dual Specht modules to cell modules, it follows that $\ind_m X$ has a cell filtration. The exactness of $\Res_l$ then ensures that $\Res_l \ind_m X$ admits a filtration of $\W_{r-l}$ modules, where each subquotient is isomorphic to $\Res_l \ind_m S'(\banu)$. By Lemma \ref{Res l of a cell module has a dual Specht filtration for type C}, each such subquotient lies in $ \F_{\W_{r-l}}(S)$. Since $\F_{\W_{r-l}}(S)$ is closed under extensions, we conclude that $\Res_l\ind_m X \in \F_{\W_{r-l}}(S)$.
\end{proof}

\begin{theorem}\label{exactness of Hom for type C}
  Let $\mathrm{char}~K \neq 2,3$, and $\delta \neq 0$. For $(l,\blam) \in \Lambda$, the functor $\Hom_{B} \big( M(l,\blam),- \big)$ is exact in $\F_B(\Theta)$. 
\end{theorem}
\begin{proof}
    Consider a short exact sequence 
\begin{equation} \label{Hom functor is exact sequence 1 for type C}
   \begin{tikzcd}
        0 \arrow[r] &C \arrow[r] & D \arrow[r] &E \arrow[r] &0
    \end{tikzcd}
\end{equation}   
   in $\F_B(\Theta)$, where $C, D, E \in \mathcal{F}_B(\Theta)$. Applying the exact functor $\Res_l$, followed by the covariant functor $\Hom_{\W_{r-l}} (M(\blam), -)$, we obtain the long exact sequence 
\begin{equation}\label{Hom functor is exact sequence 2 for type C }
 \begin{tikzcd}
        0 \arrow[r]  &\Hom_{\W_{r-l}} (M(\blam), Ce_l) \arrow[r]  &\Hom_{\W_{r-l}}(M(\blam), De_l)
        \ar[draw=none]{d}[name=X, anchor=center]{}
     \ar[rounded corners,
            to path={ -- ([xshift=2ex]\tikztostart.east)
                      |- (X.center) \tikztonodes
                      -| ([xshift=-2ex]\tikztotarget.west)
                      -- (\tikztotarget)}]{dll}[at end]{} \\     
         \Hom_{\W_{r-l}} (M(\blam), Ee_l) \arrow[r] & 
        \Ext_{\W_{r-l}}^1 (M(\blam), Ce_l) \arrow[r] &\cdots.
    \end{tikzcd}
\end{equation}  
By  Corollary \ref{Res l of any module having cell filtration has dual Specht filtration in type C}, we have $C e_l \in \F_{\W_{r-l}}(S)$. To conclude exactness, it suffices to show that $\Ext_{\W_{r-l}}^1 (M(\blam), Ce_l)=0$.
 %
 %
%
To show this, we again use Corollary \ref{Res l of any module having cell filtration has dual Specht filtration in type C} and Lemma \ref{Res l of a cell module has a dual Specht filtration for type C}. It is enough to prove that 
\begin{equation*}
    \Ext_{\W_{r-l}}^1 (M(\blam), \Res_l \ind_m S'(\banu))=0. 
\end{equation*}
Applying the Eckmann-Shapiro Lemma [\cite{CE}, Chapter XI, Theorem 3.1], we get
\begin{align*}
     \Ext_{\W_{r-l}}^1 (M(\blam), \Res_l \ind_m S'(\banu)) &\cong  \Ext_{\W_{\blam}}^1 (1_{\W_{\blam}}, \Res_{\W_{\blam}}^{\W_r}\big(\Res_l \ind_m S'(\banu)) \big)\\
     &\cong \Ext_{\W_{\blam}}^1 (1_{\W_{\blam}}, \big(S_{\lambda_{1}^{1}} \otimes \cdots \otimes  S_{\lambda_{k_1}^{1}} \big)\otimes \big(S'(\lambda_{1}^{2}) \otimes \cdots \otimes  S'(\lambda_{k_2}^{2})\big)).
\end{align*}
The last isomorphism follows from the fact that restriction of a dual Specht module from $\W_r$ to $\W_{\bolam}$ is the tensor product of dual Specht modules. If $\mathrm{char}~K \neq 2,3$, by Theorem \ref{Ext between trivial and dual Specht module is zero for hyperoctahedral group} and [\cite{HN}, Proposition 4.1.1], we have 
\[
\Ext_{K\Sg_{\lambda_{i}^{1}}}^1 (1_{K\Sg_{\lambda_{i}^{1}}}, S_{\lambda_i^{1}})=0 \text{ and } \Ext_{\W_{\lambda_{j}^{2}}}^1 (1_{\W_{\lambda_{j}^{2}}}, S'(\lambda_j^{2}))=0, \text{ for } i=1, \cdots k_1; j=1, \cdots, k_2.
\]
Hence, by [\cite{CE}, Chapter XI, Theorem 3.1], it follows that $  \Ext_{\W_{r-l}}^1 (M(\blam), \Res_l \ind_m S'(\banu))=0, $ as desired. 
\end{proof}
A $B$-module $M$ is said to be \textit{relative projective} in $\F_B(\Theta)$ if $\Ext_B^1(M, X)=0$, for all $X \in \F_B(\Theta)$.
\begin{corollary} \label{relative projective of the permutation module and its direct summand for type C}
    Let $\mathrm{char}~K \neq 2,3$, and $\delta \neq 0$. For $(l,\blam)\in \Lambda$, the permutation module $M(l,\blam)$ and all its direct summands are relative projective in $\F_B(\Theta).$ In particular, $Y(l,\blam)$ is relative projective in $\F_B(\Theta)$. 
\end{corollary}
\begin{proof}
Consider the short exact sequence in (\ref{Hom functor is exact sequence 1 for type C}) and $E=\Ind_l M(\blam)$. By Theorem \ref{exactness of Hom for type C}, applying the exact functor $\Hom_B(\Ind_l M(\blam),-)$ yields the exact sequence
\begin{equation}\label{Hom functor M(l,blam) is exact sequence 4 for type C}
 \begin{tikzcd}
        0 \arrow[r]  &\Hom_{B} (\Ind_lM(\blam), C) \arrow[r]  &\Hom_{B}(\Ind_lM(\blam), D) \arrow[r, "f"] &\End_{B} (\Ind_l M(\blam))\arrow[r]  &0.
 \end{tikzcd}      
\end{equation}
Since the map $f$ in (\ref{Hom functor M(l,blam) is exact sequence 4 for type C}) is surjective, there exists a homomorphism $\beta: \Ind_l M(\blam) \longrightarrow D$ such that the sequence in (\ref{Hom functor M(l,blam) is exact sequence 4 for type C}) is split exact. This implies $\Ext_B^1(\Ind_l M(\blam), C)=0$, for all $C \in \F_B(\Theta)$.

Let $E$ be any direct summand of $M(l,\blam)$.  Using similar arguments, we conclude that $E$ is also relative projective in $\F_B(\Theta)$. 
\end{proof}


\section{Main Theorem} \label{Main theorem for type C}

This section aims to decompose permutation modules of type C Brauer algebras into direct sums of indecomposable Young modules, following the approach of \cite{HP} for the classical Brauer algebra.
 
\begin{theorem}\label{Decomposition of the permutation module for Brauer algebras of type C for type C}
     Let $\mathrm{char}~K \neq 2,3$, and $\delta \neq 0$. For any $(l,\blam) \in \Lambda$, the permutation module $M(l,\blam)$ of $\B(C_r, \delta)$ decomposes as a direct sum of indecomposable Young modules:
     \[
     M(l, \blam) \cong Y(l,\blam) \oplus \big(\bigoplus_{\substack{(l,\blam) \geq (m, \bmu)\\ (m,\bmu) \in \Lambda}} Y(m,\bmu)^{ a_{(m,\bmu)}}\big),
     \]
     where 
     \begin{itemize}
         \item[1.] The Young module $Y(l,\blam)$ appears with multiplicity one.  
         \item[2.] The other direct summands of $M(l,\blam)$ are indexed by $(m, \bmu) \in \Lambda$ such that $(l,\blam) \geq (m, \bmu)$.     
         \item[3.] The multiplicity of $Y(m,\bmu) $ in $M(l,\blam)$ is given by $a_{(m,\bmu)}$, a non-zero integer. 
     \end{itemize}
\end{theorem}
\begin{proof}
    The cell modules $\Theta(l,\blam)$ of $B$ satisfy the two conditions given in Lemma \ref{Hom between two cell modules vanishing condition for type C} and Lemma \ref{Ext between two cell modules vanishing condition for type C}, and  by Theorem \ref{existence of stratifying system} there exists indecomposable $B$-modules $\{P(l,\blam): (l,\blam) \in \Lambda\}$ such that $$\{\big(\Theta(l,\blam), P(l,\blam)\big)\}_{(l,\blam) \in \Lambda}$$ forms a stratifying system by Proposition \ref{stratifying system of brauer algebra of type C}. It follows from [\cite{ES}, Proposition 1.3] that there exists a standardly stratified algebra (infact quasi-hereditary algebra) $A=\End_{B}(  \underset{(l,\blam)}{\bigoplus} P(l,\blam))$, where the standard module is given by $\Delta(l,\blam)=\Hom_B(\Theta(l,\blam), \underset{(l,\blam)}{\bigoplus} P(l,\blam))$ with respect to the ordering $(\Lambda, \leq^{\mathrm{op}})$, where $ \leq^{\mathrm{op}}$ is the opposite of the natural order. As a result, it follows from [\cite{ES}, Theorem 1.6] that the category $\F_{A}(\Delta)$ of $A$-modules having a filtration by $\Delta(l,\blam)$ is contravariantly equivalent to $\F_B(\Theta)$. 
    
    Since $M(l,\blam)$ has a cell filtration by Theorem \ref{cell filtration of permutation module for type C}, and $\Hom_B(M(l,\blam),-)$ is exact in $\F_B(\Theta)$ by Theorem \ref{exactness of Hom for type C}; it follows that $\Hom_B(-, \mathcal{M}(l,\blam))$ is exact in $\F_A(\Delta)$, where $\mathcal{M}(l,\blam)$ denotes the image of $M(l,\blam)$ under the equivalence. Since $A$ is standardly stratified and $\Hom_B(-, \mathcal{M}(l,\blam))$ is exact in $\F_A(\Delta)$, it follows from [\cite{DR}, Theorem 1] that $\mathcal{M}(l,\blam)^*$ is projective. Thus, we have a decomposition
    \[
    \mathcal{M}(l,\blam)^*= \bigoplus_{(r,\banu)\in \Lambda} \mathcal{P}(r,\banu)^{a_{(r,\banu)}},  
    \]
    where $\mathcal{P}(r,\banu)$ are indecomposable projective $A$-modules. Under this equivalence, this corresponds to
    \[
    M(l,\blam)^*= \bigoplus_{(r,\banu)\in \Lambda} P(r,\banu)^{ a_{(r,\banu)}}.
    \]
    Since $Y(m,\bmu)^*$ is an indecomposable summand of $M(l,\blam)^*$, there exists $(r,\banu) \in \Lambda$ such that $Y(m,\bmu)^* \cong P(r,\banu)$. By Proposition \ref{Characterization of Young module in type C}, we conclude that
    \[
    M(l,\blam)= \bigoplus_{(l,\blam)\geq (m,\bmu)} Y(m,\bmu)^{ a_{(m,\bmu)}}.
    \]
    This completes the proof.
\end{proof}
\begin{remark}
The results established in this article naturally lead to the definition of a new Schur algebra associted with $\B(C_r,\delta)$. We define the \textit{new Schur algebra} corresponding to $\B(C_r,\delta)$, denoted by $\mathcal{S}(\B(C_r,\delta))$, as 
\[
\mathcal{S}(\B(C_r,\delta)):=\End_{\B(C_r,\delta)}(Y),
\]
where $Y=\bigoplus_{(m,\bmu)\in \Lambda} Y(m,\bmu)$ is the direct sum of all Young modules $Y(m,\bmu)$ of $\B(C_r,\delta)$. By [\cite{HHKP}, Theorem 13.1], there is a Schur-Weyl duality between $\mathcal{S}(\B(C_r,\delta))$ and $\B(C_r,\delta)$. In particular, 
\[
\End_{\End_{\B(C_r,\delta)}(Y)}(Y)=\B(C_r,\delta).
\]
This duality further will help in the structural understanding of representation theory of $\B(C_r,\delta)$ and will offer a natural framework for exploring its connection with Schur algebras.
\end{remark}

\textbf{Acknowledgements} 

The first author's research is supported by IISER-Thiruvananthapuram PhD fellowship. The seconds author's research was partially supported by IISER-Thiruvananthapuram, SERB-Power Grant SPG/2021/004200 and Prof. Steffen Koenig's research grant. The second author also would like to acknowledge the Alexander von Humboldt Foundation for their support.

 \bibliographystyle{abbrv}
 \bibliography{Bibliography.bib}

\begin{thebibliography}{10}

\bibitem{Aa77}
E.~Al-Aamily.
\newblock {\em Representation theory of Weyl groups of type $B_n$}.
\newblock Ph.d. thesis, University of Wales, 1977.

\bibitem{AMP}
E.~al~Aamily, A.~O. Morris, and M.~H. Peel.
\newblock The representations of the {W}eyl groups of type {$B\sb{n}$}.
\newblock {\em J. Algebra}, 68(2):298--305, 1981.

\bibitem{Bo}
C.~Bowman.
\newblock Brauer algebras of type {$C$} are cellularly stratified.
\newblock {\em Math. Proc. Cambridge Philos. Soc.}, 153(1):1--7, 2012.

\bibitem{CanHyper}
H.~Can.
\newblock On the inequivalence and standard basis of the {S}pecht modules of
  the hyperoctahedral groups.
\newblock {\em Commun. Fac. Sci. Univ. Ank. Ser. A1 Math. Stat.},
  47(1-2):125--138, 1998.

\bibitem{CE}
H.~Cartan and S.~Eilenberg.
\newblock {\em Homological algebra}.
\newblock Princeton Landmarks in Mathematics. Princeton University Press,
  Princeton, NJ, 1999.
\newblock With an appendix by David A. Buchsbaum, Reprint of the 1956 original.

\bibitem{CGsplit}
S.~Chowdhury and G.~Thangavelu.
\newblock Comparing cohomology via exact split pairs in diagram algebras.
\newblock {\em Arch. Math.}, 125:79--92, 2025.

\bibitem{CGwalled}
S.~Chowdhury and G.~Thangavelu.
\newblock Permutation modules of the walled {B}rauer algebras.
\newblock {\em arXiv preprint arXiv:2503.09406}, 2025.

\bibitem{CLY}
A.~M. Cohen, S.~Liu, and S.~Yu.
\newblock Brauer algebras of type {C}.
\newblock {\em J. Pure Appl. Algebra}, 216(2):407--426, 2012.

\bibitem{CR90}
C.~W. Curtis and I.~Reiner.
\newblock {\em Methods of representation theory. {V}ol. {I}}.
\newblock Wiley Classics Library. John Wiley \& Sons, Inc., New York, 1990.
\newblock With applications to finite groups and orders, Reprint of the 1981
  original, A Wiley-Interscience Publication.

\bibitem{DJ92}
R.~Dipper and G.~James.
\newblock Representations of {H}ecke algebras of type {$B_n$}.
\newblock {\em J. Algebra}, 146(2):454--481, 1992.

\bibitem{DM02}
R.~Dipper and A.~Mathas.
\newblock Morita equivalences of {A}riki-{K}oike algebras.
\newblock {\em Math. Z.}, 240(3):579--610, 2002.

\bibitem{DR}
V.~Dlab and C.~M. Ringel.
\newblock The module theoretical approach to quasi-hereditary algebras.
\newblock In {\em Representations of algebras and related topics ({K}yoto,
  1990)}, volume 168 of {\em London Math. Soc. Lecture Note Ser.}, pages
  200--224. Cambridge Univ. Press, Cambridge, 1992.

\bibitem{Do}
S.~Donkin.
\newblock {\em The {$q$}-{S}chur algebra}, volume 253 of {\em London
  Mathematical Society Lecture Note Series}.
\newblock Cambridge University Press, Cambridge, 1998.

\bibitem{Er05}
K.~Erdmann.
\newblock Stratifying systems, filtration multiplicities and symmetric groups.
\newblock {\em J. Algebra Appl.}, 4(5):551--555, 2005.

\bibitem{ES}
K.~Erdmann and C.~S\'aenz.
\newblock On standardly stratified algebras.
\newblock {\em Comm. Algebra}, 31(7):3429--3446, 2003.

\bibitem{Grb}
J.~Grabmeier.
\newblock Unzerlegbare {M}oduln mit trivialer {Y}oungquelle und
  {D}arstellungstheorie der {S}churalgebra.
\newblock {\em Bayreuth. Math. Schr.}, (20):9--152, 1985.

\bibitem{GL}
J.~J. Graham and G.~I. Lehrer.
\newblock Cellular algebras.
\newblock {\em Invent. Math.}, 123(1):1--34, 1996.

\bibitem{RuThesis}
R.~Green.
\newblock {\em Some properties of Specht modules for the wreath product of
  symmetric groups}.
\newblock PhD thesis, University of Kent, May 2019.

\bibitem{GrHyperoctahedral}
R.~M. Green.
\newblock Hyperoctahedral {S}chur algebras.
\newblock {\em J. Algebra}, 192(1):418--438, 1997.

\bibitem{HHKP}
R.~Hartmann, A.~Henke, S.~Koenig, and R.~Paget.
\newblock Cohomological stratification of diagram algebras.
\newblock {\em Math. Ann.}, 347(4):765--804, 2010.

\bibitem{HP}
R.~Hartmann and R.~Paget.
\newblock Young modules and filtration multiplicities for {B}rauer algebras.
\newblock {\em Math. Z.}, 254(2):333--357, 2006.

\bibitem{HN}
D.~J. Hemmer and D.~K. Nakano.
\newblock Specht filtrations for {H}ecke algebras of type {A}.
\newblock {\em J. London Math. Soc. (2)}, 69(3):623--638, 2004.

\bibitem{JaB}
G.~D. James.
\newblock {\em The representation theory of the symmetric groups}, volume 682
  of {\em Lecture Notes in Mathematics}.
\newblock Springer, Berlin, 1978.

\bibitem{Ja}
G.~D. James.
\newblock Trivial source modules for symmetric groups.
\newblock {\em Arch. Math. (Basel)}, 41(4):294--300, 1983.

\bibitem{KX}
S.~K\"onig and C.~Xi.
\newblock On the structure of cellular algebras.
\newblock In {\em Algebras and modules, {II} ({G}eiranger, 1996)}, volume~24 of
  {\em CMS Conf. Proc.}, pages 365--386. Amer. Math. Soc., Providence, RI,
  1998.

\bibitem{KX3}
S.~K\"onig and C.~Xi.
\newblock Cellular algebras: inflations and {M}orita equivalences.
\newblock {\em J. London Math. Soc. (2)}, 60(3):700--722, 1999.

\bibitem{Ma}
A.~Mathas.
\newblock Tilting modules for cyclotomic {S}chur algebras.
\newblock {\em J. Reine Angew. Math.}, 562:137--169, 2003.

\bibitem{Mo81}
A.~O. Morris.
\newblock Representations of {Weyl} groups over an arbitrary field.
\newblock In {\em Tableaux de Young et foncteurs de Schur en alg\`ebre et
  g\'eom\'etrie - TORUN, Pologne, 1980}, number 87-88 in Ast\'erisque, pages
  267--287. Soci\'et\'e math\'ematique de France, 1981.

\bibitem{In}
I.~Paul.
\newblock Permutation modules for cellularly stratified algebras.
\newblock {\em J. Pure Appl. Algebra}, 224(11):106412, 33, 2020.

\bibitem{RY03}
H.~Rui and W.~Yu.
\newblock On the semi-simplicity of the cyclotomic {B}rauer algebras.
\newblock {\em J. Algebra}, 277(1):187--221, 2004.

\bibitem{Schur1927}
I.~Schur.
\newblock {\"U}ber die rationalen darstellungen der allgemeinen linearen
  gruppe.
\newblock {\em Sitzungsberichte der Preussischen Akademie der Wissenschaften zu
  Berlin}, pages 58--75, 1927.

\bibitem{Dieck}
T.~tom Dieck.
\newblock Quantum groups and knot algebra.
\newblock Lecture notes, May 2004.
\newblock May 4, 2004.

\bibitem{Xi00}
C.~Xi.
\newblock On the representation dimension of finite dimensional algebras.
\newblock {\em J. Algebra}, 226(1):332--346, 2000.

\end{thebibliography}
\end{document}